\documentclass[11pt]{amsart}
\usepackage[margin=1in]{geometry}
\usepackage{amsmath}
\usepackage{amsfonts,amsmath}
\usepackage{paralist}
 \numberwithin{equation}{section}
\newcommand*{\avint}{\mathop{\ooalign{$\int$\cr$-$}}}

\newcommand{\bdz}{B_{\delta_0}(0)}

\newcommand{\ibd}{\int_{B_{\delta_0}(0)}}
\newcommand{\io}{\int_{\Omega}}

\newcommand{\iqh}{\int_{Q_{\frac{r}{2}}(z)}}

\newcommand{\ibz}{\int_{B_1(0)}}
\newcommand{\ibh}{\int_{B_{\frac{r}{2}}(y)}}
\newcommand{\ibk}{\int_{B_{r_k}(y_k)}}
\newcommand{\iqk}{\int_{Q_{r_k}(z_k)}}
\newcommand{\mzk}{m_{z_k,r_k}}

\newcommand{\qzr}{Q_r(z) }

\newcommand{\qzrk}{Q_{r_k}(z_k) }

\newcommand{\ezr}{E_r(z) }

\newcommand{\ezk}{E_{r_k}(z_k) }
\newcommand{\byr}{B_r(y) }

\newcommand{\air}{\avint_{B_\rho(y)}  }
\newcommand{\aizr}{\avint_{Q_\rho(z)}  }

\newcommand{\aiR}{\avint_{B_R(y)}  }

\newcommand{\ykr}{y_k+r_ky}
\newcommand{\tkr}{\tau_k+r_k^2\tau}

\newcommand{\mat}{\max_{t\in[\tau-\frac{1}{2}r^2, \tau+\frac{1}{2}r^2]}}
\newcommand{\maxd}{\max_{\tau\in[-\frac{1}{2}\delta^2, \frac{1}{2}\delta^2]}}
\newcommand{\maxn}{\max_{\tau\in[-\frac{1}{2}, \frac{1}{2}]}}

\newcommand{\maxth}{\max_{t\in[\tau-\frac{1}{8}r^2, \tau+\frac{1}{8}r^2]}}
\newcommand{\mnp}{(m\cdot\nabla p) }
\newcommand{\nsk}{(n_k\cdot\nabla \sk) }
\newcommand{\ot}{\Omega_T }
\newcommand{\mrt}{m_{y,r}(t) }
\newcommand{\mz}{m_{z,r} }
\newcommand{\prt}{p_{y,r}(t) }
\newcommand{\pkt}{p_{y_k,r_k}(\tkr) }
\newcommand{\wks}{|w_{k}|^2 }
\newcommand{\sk}{\psi_{k} }
\newtheorem{thm}{Theorem}[section]
\newtheorem{cor}{Corollary}[section]
\newtheorem{prop}{Proposition}[section]
\newtheorem{lem}{Lemma}[section]
\newtheorem{clm}{Claim}[section]
\begin{document}
	\title[Partial regularity for weak solutions to a PDE system]{Partial regularity of weak solutions to a PDE system with cubic nonlinearity}
	\author{Jian-Guo Liu and Xiangsheng Xu}\thanks
	{Liu's address: Department of Physics and Department of Mathematics, Duke University, Durham, NC 27708. {\it Email:} jliu@phy.duke.edu.\\
		\hspace{.2in}
		Xu's address:	Department of Mathematics and Statistics, Mississippi State
		University, Mississippi State, MS 39762.
		{\it Email}: xxu@math.msstate.edu. J. Differential Equations, to appear.}
	\begin{abstract} In this paper we investigate regularity properties of weak solutions to a PDE system that arises in the study of biological transport networks. The system consists of a possibly singular elliptic equation for the scalar pressure of the underlying biological network coupled to a diffusion equation for the conductance vector of the network. There are several different types of nonlinearities in the system. Of particular mathematical interest is a term that is a polynomial function of solutions and their partial derivatives and this polynomial function has degree three.  That is, the system contains a cubic nonlinearity. Only weak solutions to the system have been shown to exist. The regularity theory for the system remains fundamentally incomplete. In particular, it is not known whether or not weak solutions develop singularities.  In this paper we obtain a partial regularity theorem, which
	gives an estimate for the parabolic Hausdorff dimension of the set of possible singular points. 
	\end{abstract}
	\maketitle
		\section{Introduction}
Let $\Omega$ be a bounded domain in $\mathbb{R}^N$ and $T$ a positive number. Set $\ot=\Omega\times(0,T)$. We study the behavior of weak solutions of the system
	\begin{eqnarray}
	-\mbox{{div}}\left[(I+m\otimes m)\nabla p\right]&=& S(x)\ \ \ \mbox{in $\ot$,}\label{e1}\\
	\partial_tm-D^2\Delta m-E^2\mnp\nabla p+|m|^{2(\gamma-1)}m&=& 0\ \ \ \mbox{in $\ot$}\label{e2}
		\end{eqnarray}
		for given function $S(x)$ and physical parameters $D, E, \gamma$ with properties:
		\begin{enumerate}
			\item[(H1)] $S(x)\in L^q(\Omega), \ q>\frac{N}{2}$; and
			\item[(H2)] $D, E\in (0, \infty), \gamma\in (\frac{1}{2}, \infty)$.
		\end{enumerate}
		This system has been proposed by Hu and Cai (\cite{H}, \cite{HC}) to describe 
		natural network formulation. Then the scalar  pressure function $p=p(x,t)$  follows  Darcy's law, while the vector-valued function $m=m(x,t)$ is the conductance vector. The function $S(x)$ is the time-independent source term. Values of the parameters $D, E$, and $\gamma$ are
		determined by the particular physical applications one has in mind. For example, $\gamma =1$ corresponds to leaf venation \cite{H}. 
		Of particular physical interest is the initial boundary value problem: in addition to (\ref{e1}) and (\ref{e2}) one requires
		\begin{eqnarray}
		m(x,0)&=&m_0(x),\ \ \ \ x\in\Omega,\label{e3}\\
	p(x,t)=0, && m(x,t)=0, \ \ \ (x,t)\in \Sigma_T\equiv\partial\Omega\times(0,T),\label{e4}
		\end{eqnarray}
		at least in a suitably weak sense; here the initial data should satisfy
		$$m_0(x)=0 \ \ \ \mbox{on $\partial\Omega$.}$$
		The existence of weak solutions of this initial boundary value problem was proved by Haskovec, Markowich, and Perthame \cite{HMP}. However, the regularity theory remains fundamentally incomplete. In particular, it is not known whether or not weak solutions develop singularities.
		
		Let us call a point $(x,t)\in \ot$ singular if $m$ is not H\"{o}lder continuous in any neighborhood of $(x,t)$; the remaining points will be called regular points. By a partial regularity theorem, we mean an estimate for the dimension of the set $S$ of singular points. It is well-known that weak solutions to even
		uniformly elliptic systems of partial differential equations are not regular everywhere. We refer the reader to \cite{G} for counter examples. Thus it is only natural to seek partial regularity theorems for these weak solutions. The system under our consideration exhibits a rather peculiar nonlinear structure. The first equation in the system degenerates in the $t$-variable and the elliptic coefficients there are singular in the sense that
		they are not uniformly bounded above a priori, while the second equation contains the term $\mnp\nabla p$, which is a cubic nonlinearity.
		Thus
		the classical partial regularity argument developed in (\cite{G}, \cite{CKN}) does not seem to be applicable here. Our system does resemble the so-call thermistor problem considered in (\cite{X1}-\cite{X3}). The key difference is that the elliptic coefficients in the preceding papers and also in \cite{G} are assumed to be bounded and continuous functions of solutions. As a result, the modulus of continuity can be
		taken to be a bounded, continuous, and concave function. This fact is essential to the arguments in both \cite{X1} and \cite{G}. Our elliptic coefficients here are quadratic in $m$, and thus a new proof must be developed.

					\noindent {\bf Definition.} A pair $(m, p)$ is said to be a weak solution if:
					\begin{enumerate}
						\item[(D1)] $m\in L^\infty(0,T; \left(W^{1,2}_0(\Omega)\cap L^{2\gamma}(\Omega)\right)^N), \partial_tm\in L^2(0,T; \left(L^2(\Omega)\right)^N), p\in L^\infty(0,T; W^{1,2}_0(\Omega)), \\  m\cdot\nabla p \in L^\infty(0,T;  L^{2}(\Omega))$;
						\item[(D2)] $m(x,0)=m_0$ in $C([0,T]; \left(L^2(\Omega)\right)^N)$;
						\item[(D3)] (\ref{e1}) and (\ref{e2}) are satisfied in the sense of distributions.
					\end{enumerate}
					A result in \cite{HMP} asserts that (\ref{e1}) 
					-(\ref{e4}) has a weak solution provided that, in addition to assuming $S(x)\in L^2(\Omega)$ and (H2), we also have
					\begin{enumerate}
						\item[(H3)] $m_0\in\left( W^{1,2}_0(\Omega)\cap L^{2\gamma}(\Omega)\right)^N$.
					\end{enumerate} 
					Note that the question of existence in the case where $\gamma=\frac{1}{2}$ is addressed in \cite{HMPS}. In this case the term $|m|^{2(\gamma-1)}m$ is not continuous at $m=0$. 
					It must be replaced by the following function 
					$$g(x,t)=\left\{\begin{array}{ll}
					|m|^{2(\gamma-1)}m & \mbox{if $m\ne 0$,}\\
					\in [-1,1]^N &\mbox{if $m\ne 0$.}
					\end{array}\right.$$
					
					Partial regularity relies on local estimates \cite{G}. One peculiar feature about our problem (\ref{e1})-(\ref{e4}) is that certain important global estimates have no local versions. This is another source of difficulty for our mathematical analysis. 
					We are ready to state our main result:
					\begin{thm}
					Let (H1)-(H3) be satisfied. Assume that $N\leq 3$. Then the initial boundary value problem (\ref{e1})-(\ref{e4}) has a weak solution on $\ot$ whose singular set $S$
					satisfies
					\begin{equation}
					\mathcal{P}^{N+\varepsilon}(S)=0
					\end{equation}
					for each $\varepsilon>0$.
					\end{thm}
					Here $\mathcal{P}^{s}, s\geq 0,$ denotes the $s$-dimensional parabolic Hausdorff measure. Recall
					that the s-dimensional parabolic Hausdorff measure of a set $E\subset\mathbb{R}^N\times\mathbb{R}$ is defined as follows:
						$$\mathcal{P}^s(E)=\sup_{\varepsilon>0}\inf\{\sum_{j=0}^{\infty}r_j^s: \cup_{j=0}^\infty Q_{r_j}(z_j)\supset E, r_j<\varepsilon\},$$
						where $Q_{r_j}(z_j)$ are parabolic cylinders with geometric centers  at $z_j=(y_j, \tau_j)$, i.e., one has
							$$Q_{r_j}(z_j)	=B_{r_j}(y_j)\times(\tau_j-\frac{1}{2}r_j^2,\tau_j+\frac{1}{2}r_j^2 )$$
						with	 $$B_{r_j}(y_j)=\{x\in \mathbb{R}^N:|x-y_j|<r_j\}.$$ 
						It is not difficult to see that $\mathcal{P}^s$ is an outer measure, for which all Borel sets are measurable; on its $\sigma$-algebra of measurable sets, $\mathcal{P}^k$ is a Borel regular measure (cf. \cite{F}, Chap.2.10).
						If $\mathcal{P}^s(E)<\infty$, then $\mathcal{P}^{s+\varepsilon}(E)=0$ for each $\varepsilon>0$.
						We define the parabolic Hausdorff dimension $\mbox{dim}_\mathcal{P}E$ of a set $E$ to be
						$$\mbox{dim}_\mathcal{P}E=\inf\{s\in\mathbb{R}^+:\mathcal{P}^s(E)=0\}.$$
						Then Theorem 1.1 says that 
						\begin{equation}
						\mbox{dim}_\mathcal{P}S\leq N.
						\end{equation}

					Hausdorff measure $\mathcal{H}^s$ is defined in an entirely similar manner, but with $Q_{r_j}(z_j)$ replaced by an arbitrary closed subset of $\mathbb{R}^N\times\mathbb{R}$ of diameter at most $r_j$. (One usually normalizes $\mathcal{H}^s$ for integer $s$ so that it agrees with surface area on smooth $s$-dimensional surfaces.)
						 Clearly,
						\begin{equation}
						\mathcal{H}^{k}(X)\leq c(k)\mathcal{P}^k(X)\ \ \ \mbox{for each $X\subset\mathbb{R}^N\times\mathbb{R}$. }
						\end{equation}	
						To characterize the singular set $S$, we will need to invoke the following known result.
							\begin{lem}
								Let $f\in L^1_{\mbox{loc}}(\Omega_T)$ and for $0\leq s<N+2$ set
								$$E_s=\{z\in\Omega_T:\limsup_{\rho\rightarrow 0^+}\rho^{-s}\int_{Q_\rho(z)}|f|dxdt>0\}.$$
								Then $\mathcal{P}^s(E_s)=0$.
							\end{lem}
							The proof of this lemma is essentially contained in \cite{CKN}.
							
							A key observation about our weak solutions in the study of partial regularity is the following proposition.
							\begin{prop}\label{prop1.1}
								Let \textup{(H1)-(H3)} be satisfied and $(p, m)$ be a weak solution to (\ref{e1})-(\ref{e4}). Then we have
								\begin{equation}
								p\in C([0,T]; L^2(\Omega)).
								\end{equation}
							\end{prop}
							
							The proof of this proposition will be given at the end of Section 2.

					Let $(m, p)$ be a weak solution. In view of (\cite{E},\cite{X1}), to establish Theorem 1.1, we will need to define a suitable scaled energy $E_r(z)$ for our system. For this purpose,
					let  $z=(y,\tau)\in \Omega_T, r>0$ with $\qzr\subset\ot$ and pick
					\begin{equation}
						0<\beta<\min\{2-\frac{N}{q}, 1\},\label{cb}
					\end{equation} where $q$ is given as in (H1). We consider the following quantities:
					\begin{eqnarray}
					\prt&=&\avint_{B_r(y)} p(x,t)dx=\frac{1}{|\byr|}\int_{B_r(y)} p(x,t)dx,\\
					\mz &=&\avint_{Q_\rho(z)}  m(x,t)dxdt,\\
				A_r(z)&=&\frac{1}{r^{N}}\mat\int_{B_r(y)}(p(x,t)-\prt)^2dx.
						\end{eqnarray}
					The right choice for $E_r(z)$ seems to be 
					\begin{equation}
					\ezr=\frac{1}{r^{N+2}}\int_{Q_r(z)}|m-\mz|^2dxdt+A_r(z)+r^{2\beta}.
					\end{equation}
					The last term in $E_r(z)$ accounts for the non-homogeneous term $S(x) $ in (\ref{e1}). 
					 Due to the fact that the first equation (\ref{e1}) does not have
					 the $\partial_t p$ term, we are forced to use the term $A_r(z)$ instead of $\frac{1}{r^{N+2}}\int_{Q_r(z)}|p-p_{z,r}|^2dxdt$ in $E_r(z)$. This will cause two problems: one is that in our application of the classical blow-up argument (\cite{E},\cite{G},\cite{X1}), the resulting blow-up sequence is not compact in the desired function space; the other is the characterization of the singular set $S$. That is, it is not immediately clear how one can describe the set
					 \begin{equation}
					 \ot\setminus\{z\in\ot:\lim_{r\rightarrow 0}A_r(z)=0\}\label{rm11}
					 \end{equation} 
					 in terms of the parabolic Hausdorff measure. (Note that this issue is rather simple in the context of \cite{X1}.) To overcome  these two problems, we find a suitable decomposition of $p$. This enables us to show that the lack of compactness in the blow-up sequence does not really matter. To be more specific, we obtain that the blow-up sequence can be decomposed into the sum of two other sequences, one of which converges strongly while the terms of the other are
					 very smooth in the space variables, and this is good enough for our purpose. This idea was first employed in \cite{X1}. However, as we mentioned earlier, the nature of our mathematical difficulty here is totally different.  A similar decomposition technique can also be used to derive the parabolic Hausdorff dimension of the set in (\ref{rm11}). 
					 
				The key to our development is this assertion about energy:
				\begin{prop}\label{pro12}Let the assumptions of Theorem 1.1 hold.
					For each $M>0$ there exist constants $0<\varepsilon,\delta<1$ such that\begin{equation}
					|\mz|	\leq M\ \ \mbox{and}\ \ \ezr\leq\varepsilon\label{con1}
					\end{equation}
					imply
					\begin{equation}
					E_{\delta r}(z)\leq \frac{1}{2}\ezr\label{con2}
					\end{equation}
					for all $z\in\ot$ and $r>0$ with $\qzr\subset\ot$.
				\end{prop}	
				The proof of this proposition is given in Section 4. It relies on the decomposition of
			the function $p$ we mentioned earlier.
				An immediate consequence of this proposition is:
				\begin{cor}Let the assumptions of Theorem 1.1 hold.
					To each $M>0$ there corresponds a pair of numbers $\delta_1,\varepsilon_1$ in $(0,1)$
					such that
				  whenever
					\begin{equation}
					|\mz|<\frac{M}{2}\ \mbox{and} \ E_r(z)<\varepsilon_1\label{js2}
					\end{equation}
					we have
					\begin{equation}
					E_{\delta_1^kr}(z)\leq \left(\frac{1}{2}\right)^k\varepsilon_1\ \ \ \mbox{for each positive integer $k$}.\label{js1}
					\end{equation}
				\end{cor}
				\begin{proof}
					We essentially follow the proof of Corollary 3.8 in \cite{X1} (also see \cite{G}). Let $M>0$ be given. By Proposition 1.2, there exist $0<\varepsilon,\delta<1$ such that (\ref{con1})	and (\ref{con2}) hold. We claim that we
					can take
					\begin{eqnarray}
					\delta_1&=&\delta,\\
					\varepsilon_1&=&\min\left\{\varepsilon,\left(\frac{M\delta^{N+2}(\sqrt{2}-1)\sqrt{\omega_N}}{2\delta^{N+2}+2\sqrt{2}}\right)^2\right\},
					\end{eqnarray}
					where $\omega_N$ is the volume of the unit ball in $\mathbb{R}^N$.
					To see this,  
					let (\ref{js2}) hold. Obviously, (\ref{js1}) is satisfied for $k=1$. Now for each positive integer $j$ suppose (\ref{js1}) is true for all $k\leq j$. We will show that it is also true for $k=j+1$. To this end, we integrate the inequality 
					$$|m_{z, \delta^ir}-m_{z, \delta^{i-1}r}|\leq |m_{z, \delta^ir}-m(x,t)|+|m(x,t)-m_{z, \delta^{i-1}r}|$$
					over $Q_{\delta^ir}(z)$ to derive
					\begin{eqnarray}
					|m_{z, \delta^ir}-m_{z, \delta^{i-1}r}|&\leq&\avint_{Q_{\delta^ir}(z)} |m_{z, \delta^ir}-m(x,t)|dxdt\nonumber\\
					&&+\frac{1}{\delta^{N+2}}\avint_{Q_{\delta^{i-1}r}(z)}|m(x,t)-m_{z, \delta^{i-1}r}|dxdt\nonumber\\
					&\leq&\left(\avint_{Q_{\delta^ir}(z)} |m_{z, \delta^ir}-m(x,t)|^2dxdt\right)^\frac{1}{2}\nonumber\\
					&&+\frac{1}{\delta^{N+2}}\left(\avint_{Q_{\delta^{i-1}r}(z)}|m(x,t)-m_{z, \delta^{i-1}r}|^2dxdt\right)^\frac{1}{2}\nonumber\\
					&\leq&\left(\frac{1}{\omega_N}E_{\delta^ir}(z)\right)^\frac{1}{2}+\frac{1}{\delta^{N+2}}
					\left(\frac{1}{\omega_N}E_{\delta^{i-1}r}(z)\right)^\frac{1}{2}\nonumber\\
					&\leq&\left(\frac{1}{\omega_N}\left(\frac{1}{2}\right)^i\varepsilon_1\right)^\frac{1}{2}+\frac{1}{\delta^{N+2}}
					\left(\frac{1}{\omega_N}\left(\frac{1}{2}\right)^{i-1}\varepsilon_1\right)^\frac{1}{2},\nonumber\\
					&&i=1,\cdots, j.
					\end{eqnarray}
					Subsequently, we have
					\begin{eqnarray}
					|m_{z, \delta^jr}|&\leq&\mz+\sum_{i=1}^{j}|m_{z, \delta^ir}-m_{z, \delta^{i-1}r}|\nonumber\\
					&\leq&\frac{M}{2}+\sum_{i=1}^{j}\left(\frac{1}{\omega_N}\left(\frac{1}{2}\right)^i\varepsilon_1\right)^\frac{1}{2}+\sum_{i=1}^{j}\frac{1}{\delta^{N+2}}
					\left(\frac{1}{\omega_N}\left(\frac{1}{2}\right)^{i-1}\varepsilon_1\right)^\frac{1}{2}\nonumber\\
					&\leq& \frac{M}{2}+\frac{\delta^{N+2}+\sqrt{2}}{\delta^{N+2}(\sqrt{2}-1)\sqrt{\omega_N}}\sqrt{\varepsilon_1}\leq M.
					\end{eqnarray}
					By Proposition 1.2, (\ref{js1})	 holds for $k=j+1$. This completes the proof.
					\end{proof}
				This corollary combined with the argument in (\cite{G}, p.86) asserts that there exist $c=c(\delta_1, \varepsilon_1, r)	\in (0,1), \gamma=\gamma(\delta_1)>0$ such that
				\begin{equation}\label{js3}
				E_\rho(z)\leq c\rho^\gamma\ \ \ \mbox{for all $0<\rho\leq r$.}
				\end{equation}
				Obviously, $\mz, \avint_{Q_r(z)}|m-m_{z,r}|^2dxdt$ are both continuous functions of $z$. By Proposition 1.1, $\ezr$ is also a continuous function of $z$. Thus whenever (\ref{js2}) holds for some $z=z_0$ there is an open neighborhood $O$ of $z_0$ over which (\ref{js2}) remains true. As a result, (\ref{js3})
				is satisfied on $ O$. This puts us in a position to apply a result in \cite{S1}. To state the result, we 
			define, for $\mu\in (0,1)$,
				$$[m]_{\mu, O}=\sup\left\{\frac{|m(x,t)-m(y,\tau)|}{\left(|x-y|+|t-\tau|^{\frac{1}{2}}\right)^\mu}: (x,t), (y,\tau)\in O\right\}.$$
				Parabolic H\"{o}lder spaces can be characterized by the following version of Campanato's
				theorem (\cite{S1}, Theorem 1).
				\begin{lem} Let $u\in L^2(\Omega_T)$.
					If there exist $\alpha\in(0,1)$ and $R_0>0$ such that
					$$\avint_{Q_\rho(z)}|u-u_{z,\rho}|^2dxdt\leq A^2\rho^{2\alpha}$$
					for all $z$ in an open subset $O$ of $\Omega_T$ and all
					$\rho\leq R_0$ with $Q_\rho(z)\subset\ot$, then we have
					$$[m]_{\alpha, O}\leq c(N)A.$$
					That is,  $u$ is H\"{o}lder continuous in $O$.
				\end{lem}
					
					To describe the singular set $S$, we set
					\begin{equation}
					R=\{z=(y,\tau)\in \ot:\sup_{r>0}|\mz|<\infty, \ \lim_{r\rightarrow 0} E_r(z)=0\}.\label{rt66}
					\end{equation}
					Here and in what follows $\lim_{r\rightarrow 0}$ means $\lim_{r\rightarrow 0^+}$ because we always have $r>0$.
					If $z\in R$, we take $M>2\sup_{r>0}|\mz|$. By Corollary 1.1, there exist $\delta_1,\varepsilon_1\in (0,1)$ such that (\ref{js2})	and (\ref{js1}) hold. We can find a $r$ such that
				$$\ezr<\varepsilon_1.$$ For the same $r$ we obviously have 
				$$|\mz|<\frac{M}{2}.$$ Consequently, $m$ is H\"{o}lder continuous in a neighborhood of $z$. That is, $R$ is a set of regular points. Obviously, $R$ is an open set. 
				
				Note that since we have the term $A_r(z)$ instead of $\frac{1}{r^{N+2}}\int_{Q_r(z)}|p-p_{z,r}|^2dxdt$ in $E_r(z)$ Proposition 1.2 does not imply that $p$ is locally H\"{o}lder continuous in the space-time domain $R$. The difference between the two quantities can be seen from the following calculation:
				\begin{eqnarray}
				\avint_{Q_\rho(z)} |p-p_{z,r}|^2 dxdt&=&\frac{1}{r^2}\int_{\tau-\frac{1}{2}r^2}^{\tau+\frac{1}{2}r^2}\avint_{B_r(y)}|p-p_{z,r}|^2 dxdt\nonumber\\
				&\leq&\frac{2}{r^2}\int_{\tau-\frac{1}{2}r^2}^{\tau+\frac{1}{2}r^2}\avint_{B_r(y)}|p-p_{y,r}(t)|^2 dxdt\nonumber\\
				&&+\frac{2}{r^2}\int_{\tau-\frac{1}{2}r^2}^{\tau+\frac{1}{2}r^2}|p_{y,r}(t)-p_{z,r}|^2 dt\nonumber\\
				&\leq&\frac{2}{\omega_N}A_r(z)+\frac{2}{r^2}\int_{\tau-\frac{1}{2}r^2}^{\tau+\frac{1}{2}r^2}|p_{y,r}(t)-p_{z,r}|^2 dt.
				\end{eqnarray} Obviously, the last term above causes the problem. Of course, for each $t=t_0$, $p(x,t_0)$ is locally H\"{o}lder continuous in $x$ in $R\cap\{t=t_0\}$.
				
				To estimate the parabolic Hausdorff dimension of the singular set $S\subseteq \ot\setminus R$, we have the following proposition.
				\begin{prop}\label{prop1.3}
					Let \textup{(H1)-(H3)} hold and $(p, m)$ be a weak solution. Then we have
					\begin{equation}
					\mbox{dim}_\mathcal{P}(\ot\setminus R)=N.\label{rt67}
					\end{equation}
				\end{prop}
				
				The proof of this proposition relies on almost the same decomposition of $p$ as that in the proof of Proposition 1.2. The details will be given in Section 3.
				
			Thus Theorem 1.1 is a consequence of Propositions \ref{prop1.1}-\ref{prop1.3}.
			The rest of the  paper is organized as follows.  In Section 2, we develop some new global estimates. They serve as a motivation for our local estimates. The section will end with the proof of Proposition 1.1.  In Section 3, we will first establish some local estimates and then proceed to prove  Proposition 1.3. Section 4 is devoted to the proof of Proposition 1.2. Note that the three propositions are independent, and thus the order of their
			proofs is not important.
				
				\section{Global Estimates}
		 In this section, we first summarize the main a priori estimates already established in \cite{HMP}. Then we present our new global estimates. The proof of Proposition 1.1 is given at the end.
		
		To begin with, we use $p(x,t)$ as a test function in (\ref{e1}) to obtain
		\begin{equation}
		\int_{\Omega} |\nabla p|^2dx+\int_{\Omega} \mnp^2dx=\int_{\Omega} S(x)pdx.\label{f1}
		\end{equation}
		Here and in what follows we suppress the dependence of $p, m$ on $(x,t)$ for simplicity of notation if no confusion arises.
		Let $\tau\in (0,T), \ \Omega_\tau=\Omega\times (0,\tau)$. Take the dot product of both sides of (\ref{e2}) with $m$, integrate the resulting equation over $\Omega_\tau$, and thereby yield
		\begin{eqnarray}
		\lefteqn{\frac{1}{2}\int_{\Omega}|m(x,\tau)|^2dx+D^2\int_{\Omega_\tau}|\nabla m|^2dxdt}\nonumber\\
		&&-E^2\int_{\Omega_\tau} \mnp^2dxdt+\int_{\Omega_\tau}|m|^{2\gamma}dxdt=	\frac{1}{2}\int_{\Omega}|m_0|^2dx,\label{f2}
		\end{eqnarray}
		where $|\nabla m|^2=|\nabla\otimes m|^2=\sum_{i,j=1}^{N}(\frac{\partial m_j}{\partial x_i})^2$.
		Multiply through (\ref{f1}) by $2E^2$, integrate over $(0,\tau)$, and then add it to (\ref{f2}) to arrive at
		\begin{eqnarray}
			\lefteqn{\frac{1}{2}\int_{\Omega}|m(x,\tau)|^2dx+D^2\int_{\Omega_\tau}|\nabla m|^2dxdt+E^2\int_{\Omega_\tau} \mnp^2dxdt}\nonumber\\
			&&+\int_{\Omega_\tau}|m|^{2\gamma}dxdt+2E^2	\int_{\Omega_\tau} |\nabla p|^2dxd\tau\nonumber\\
			&=&\frac{1}{2}\int_{\Omega}|m_0|^2dx+2E^2\int_{\Omega_\tau} S(x)pdxdt.\label{f3}
		\end{eqnarray}
		Take the dot product of (\ref{e2}) with $\partial_tm$ and integrate the resulting equation over $\Omega$ to obtain
		\begin{eqnarray}
		\lefteqn{\int_{\Omega}|\partial_tm|^2dx+\frac{D^2}{2}\frac{d}{dt}\int_{\Omega}|\nabla m|^2dx}\nonumber\\
		&&-E^2\int_{\Omega}\mnp \nabla p\partial_tm dx+\frac{1}{2\gamma}\frac{d}{dt}\int_{\Omega}|m|^{2\gamma}dx=0.\label{f4}
		\end{eqnarray}
		Use $\partial_tp$ as a test function in (\ref{e1}) to derive
		\begin{equation}
		\frac{1}{2}\frac{d}{dt}\int_{\Omega}|\nabla p|^2dx+\int_{\Omega}\mnp m\nabla\partial_tp dx=\int_{\Omega} S(x)\partial_tp dx.
		\end{equation}
		Multiply through this equation by$-E^2$ and add the resulting one to (\ref{f4}) to obtain
		\begin{eqnarray}
		\lefteqn{\int_{\Omega}|\partial_tm|^2dx+\frac{D^2}{2}\frac{d}{dt}\int_{\Omega}|\nabla m|^2dx-\frac{E^2}{2}\frac{d}{dt}\int_{\Omega}|\mnp^2 dx}\nonumber\\
		&&-\frac{E^2}{2}\frac{d}{dt}\int_{\Omega}|\nabla p|^2dx+\frac{1}{2\gamma}\frac{d}{dt}\int_{\Omega}|m|^{2\gamma}dx=-E^2\int_{\Omega} S(x)\partial_tp dx.
		\end{eqnarray}
		Differentiate (\ref{f1}) with respect to $t$, multiply through the resulting equation by $E^2$, then add it to the above equation, and thereby deduce
		\begin{eqnarray}
		\lefteqn{\int_{\Omega_\tau}|\partial_tm|^2dxdt+\frac{D^2}{2}\int_{\Omega}|\nabla m(x,\tau)|^2dx+\frac{E^2}{2}\int_{\Omega}\mnp^2 dx}\nonumber\\
		&&+\frac{E^2}{2}\int_{\Omega}|\nabla p|^2dx+\frac{1}{2\gamma}\int_{\Omega}|m|^{2\gamma}dx\nonumber\\
		&=&\frac{D^2}{2}\int_{\Omega}|\nabla m_0|^2dx+\frac{E^2}{2}\int_{\Omega}(m_0\cdot \nabla p_0)^2dx+\frac{1}{2\gamma}\int_{\Omega}|m_0|^{2\gamma}dx\nonumber\\
		&&+\frac{E^2}{2}\int_{\Omega}|\nabla p_0|^2dx,\label{f5}
		\end{eqnarray}
		where $p_0$ is the solution of the boundary value problem
		\begin{eqnarray}
		-\mbox{div}[(I+m_0\otimes m_0)\nabla p_0] &=& S(x),
		 \ \ \ \mbox{in $\Omega$,}\\
		 p_0&=& 0\ \ \ \mbox{on $\partial\Omega$.}
		\end{eqnarray}
		
 Local versions of (\ref{f1}) and (\ref{f3}) will be established in Section 3. Unfortunately, they are not enough to yield a partial regularity result. Naturally, one tries to seek  a local version
		of (\ref{f5}). But this cannot be done because we have no control over $\partial_t p$. 
		To partially circumvent this, we have developed  some new estimates.
		\begin{prop}\label{l1}
			Let (H1) and (H2) be satisfied and $(m, p)$ a weak solution of (\ref{e1})-(\ref{e4}). Then:
			\begin{enumerate}
				\item[(C1)] There is a positive number $c=c(\Omega, N)$ such that  $\|p\|_{\infty,\Omega_T}\equiv\mbox{ess sup}_{\Omega_T}|p|\leq c\|S(x)\|_{q, \Omega}$, where $\|\cdot\|_{q, \Omega}$ denotes the norm in $L^q(\Omega)$.
				We shall write $\|\cdot\|_s$ for $\|\cdot\|_{s.\Omega}$ for simplicity;
				\item[(C2)] For each $K>0$ we can choose $\beta\in (0,1)$ suitably small such that
				\begin{eqnarray*}
				\lefteqn{\int_{\Omega} \int_{0}^{|m(x,\tau)|^2}[(s-K^2)^++K^2]^\beta dsdx+\int_{\Omega_\tau} v^\beta|\nabla m|^2dxdt}\nonumber\\
				&&+\int_{\Omega_\tau} v^{\beta-1}|\nabla v|^2dxdt+\int_{\Omega_\tau}|m|^{2\gamma}v^\beta dxdt\nonumber\\
				&&+\int_{\Omega_\tau} v^\beta|\nabla p|^2dxdt+\int_{\Omega_\tau} v^\beta\mnp^2dxdt\nonumber\\
				&\leq&c\int_{\Omega_\tau}| S(x)|v^\beta  dxdt+\int_{\Omega} \int_{0}^{|m_0|^2}[(s-K^2)^++K^2]^\beta dsdx+c\ \ \mbox{for all $\tau\in (0,T)$},
				\end{eqnarray*}
				where 
				\begin{equation}
				v=(|m|^2-K^2)^++K^2\geq K^2.\label{vd}
				\end{equation}
			\end{enumerate}
		\end{prop}
		By the Sobolev embedding theorem, we have
		$$m\in L^\infty(0,T; L^{\frac{2N}{N-2}}(\Omega)).$$
		Thus the first integral on the right-hand side of the above inequality is finite.
	\begin{proof} The proof of (C1) is standard. See, e.g., (\cite{D}, p. 131). For the reader's convenience, we shall reproduce the proof here. Let $\kappa$ be a positive number to be determined. Write
		$$\kappa_n=\kappa-\frac{\kappa}{2^n},\ \ \ A_n(t)=\{x\in\Omega: p(x,t)>\kappa_n\},\ \ \ n=0, 1,2,\cdots.$$
	Use $(p-\kappa_n)^+$ as a test function in (\ref{e1}) to deduce
	\begin{eqnarray}
	\lefteqn{\int_{\Omega}|\nabla (p-\kappa_n)^+|^2dx+\int_{\Omega}|(m\cdot\nabla(p-\kappa_n)^+)^2dx}\nonumber\\
	&=&\int_{\Omega} S(x)(p-\kappa_n)^+dx\nonumber\\
	&\leq&\left(\int_{A_n(t)}|S(x)|^{\frac{2N}{N+2}}\right)^{\frac{N+2}{2N}}\|(p-\kappa_n)^+\|_{\frac{2N}{N-2}}\nonumber\\
	&\leq&c\|S(x)\|_q|A_n(t)|^{\frac{N+2}{2N}-\frac{1}{q}}\|\nabla(p-\kappa_n)^+\|_2,
	\end{eqnarray}	
	from whence follows
	\begin{equation}
	|A_{n+1}(t)|\leq c\|S(x)\|_q^{\frac{2N}{N-2}}\frac{2^{\frac{2Nn}{N-2}}}{\kappa^{\frac{2N}{N-2}}}|A_{n}(t)|^{1+\frac{2}{N-2}(2-\frac{N}{q})}.
	\end{equation}
	By (H1), we have $\alpha\equiv \frac{2}{N-2}(2-\frac{N}{q})>0$. This enables us to apply Lemma 4.1 in (\cite{D}, p. 12) to obtain
	$$|A_\infty(t)|=0,\ \ \ \mbox{provided that $\kappa=c\|S(x)\|_q$ for some $c=c(\Omega, N)$.}$$
	This implies (C1).
	
	Let $K>0, \beta>0 $ be given and $v$ be defined as in (\ref{vd}). For $L>K$, define
	\begin{equation}\label{cutoff}
	\theta_L(s)=\left\{\begin{array}{ll}
	L^2&\mbox{if $s\geq L^2$,}\\
	s&\mbox{if $K^2<s< L^2$,}\\
	K^2 &\mbox{if $s\leq K$.}
	\end{array}\right.
	\end{equation}
	Set $v_L=\theta_L(|m|^2)$.
	Then the function $v_L^\beta m$ is a legitimate test function for (\ref{e2}). Upon using it, we arrive at
	\begin{eqnarray}
	\lefteqn{\frac{1}{2}\frac{d}{dt}\int_{\Omega} \int_{0}^{|m|^2}[\theta_L(s)]^\beta dsdx+D^2\int_{\Omega} v_L^\beta|\nabla m|^2dx}\nonumber\\
	&&+\frac{D^2\beta}{2}\int_{\Omega} v_L^{\beta-1}|\nabla v_L|^2+\int_{\Omega}|m|^{2\gamma}v_L^\beta dx\nonumber\\
	&=&E^2\int_{\Omega} v_L^\beta\mnp^2dx.\label{n1}
		\end{eqnarray}
		In the derivation of the third term above,  we have used the fact that
		\begin{equation}\label{rs11}
		\nabla v_L=0\ \ \mbox{on the set where $|m|^2>L^2$ or $|m|^2<K^2$}.
		\end{equation}
		Use $v_L^\beta p$ as a test function in (\ref{e1}) to deduce
		\begin{eqnarray}
		\lefteqn{\int_{\Omega} v_L^\beta|\nabla p|^2dx+\int_{\Omega} v_L^\beta\mnp^2dx}\nonumber\\
		&=&-\int_{\Omega} \nabla p p\beta v_L^{\beta-1}\nabla v_L dx\nonumber\\
		&&-\int_{\Omega} \mnp mp\beta v_L^{\beta-1}\nabla v_L dx+\int_{\Omega} S(x)v_L^\beta p dx\nonumber\\
		&\leq & \varepsilon \beta\int_{\Omega} v_L^{\beta-1}|\nabla v_L|^2dx+ c(\varepsilon)\beta\int_{\Omega} v_L^{\beta-1}p^2|\nabla p|^2 dx\nonumber\\
		&&+\varepsilon\int_{\Omega} v_L^\beta\mnp^2dx+c(\varepsilon)\beta^2\int_{\Omega} v_L^{\beta-2}|m|^2p^2|\nabla v_L|^2dx\nonumber\\
		&&+\int_{\Omega} S(x)v_L^\beta p dx,\ \ \ \varepsilon>0.
		\end{eqnarray}
		By virtue of (\ref{rs11}), we have that $v_L^{\beta-2}|m|^2|\nabla v_L|^2=v_L^{\beta-1}|\nabla v_L|^2$. Remember that $\beta\in (0,1)$. This gives $ v_L^{\beta-1}p^2\leq \|p\|_\infty^2K^{2(\beta-1)}$. Multiply through the above inequality by $2E^2$, add the resulting inequality to 
		(\ref{n1}), thereby obtain
		\begin{eqnarray}
		\lefteqn{\frac{d}{dt}\int_{\Omega} \int_{0}^{|m|^2}[\theta_L(s)]^\beta dsdx+\int_{\Omega} v_L^\beta|\nabla m|^2dx}\nonumber\\
		&&+\beta\int_{\Omega} v_L^{\beta-1}|\nabla v_L|^2dx+\int_{\Omega}|m|^{2\gamma}v_L^\beta dx\nonumber\\
		&&+\int_{\Omega} v_L^\beta|\nabla p|^2dx+\int_{\Omega} v_L^\beta\mnp^2dx\nonumber\\
		&\leq&c\beta\frac{\|p\|_\infty^2}{K^2} \int_{\Omega} v_L^\beta|\nabla p|^2dx+c\beta^2\|p\|_\infty^2\int_{\Omega} v_L^{\beta-1}|\nabla v_L|^2dx\nonumber\\
		&& +\int_{\Omega} S(x)v_L^\beta p dx.
		\end{eqnarray}
		Choosing $\beta$ sufficiently small so that the second term on the right-hand in the above inequality can be absorbed into the third term on the left-hand side there,  integrating the resulting inequality with respect to $t$, and then taking $L\rightarrow \infty$ yields (C2). The proof is complete.
			\end{proof}
			
			It turns out that 
			a local version of (C2) is possible only if $N\leq 3$. This accounts for the restriction on the space dimension in Theorem 1.1.
		
		 	At the end of this section,  we present the proof of Proposition \ref{prop1.1}. 
		 	
		 	\begin{proof}[Proof of Proposition \ref{prop1.1}]
		 	It is easy to see that $m(x,t)\in C([0,T]; \left(L^2(\Omega)\right)^N)$. By the proof of Lemma 2.3 in \cite{X4}, we can conclude that for each $t\in [0,T]$ there is a unique weak solution $p=p(x,t)$ in the space $W^{1,2}_0(\Omega)$ to (\ref{e1}) with $m(x,t)\cdot\nabla p(x,t)\in L^2(\Omega)$. Fix a $t^*$ in $[0,T]$. Let $\{t_j\}$ be a sequence in $ [0,T]$ with the property
		 	\begin{equation}
		 	t_j\rightarrow t^*.
		 	\end{equation}
		 	Set
		 	$m_j=	m(x, t_j)$ and denote by $ p_j$ the solution of (\ref{e1}) with
		 	$m$ being replaced by $m_j$.
		 	Obviously,  we have
		 	\begin{equation}\label{r111}
		 	m_j\rightarrow m^*\equiv m(x,t^*) \ \ \ \mbox{strongly in $\left(L^2(\Omega)\right)^N$ as $j\rightarrow \infty$.}
		 	\end{equation}
		 	We claim that we also have
		 	\begin{equation}
		 	p_j\rightarrow p^*\equiv p(x,t^*),\ \mbox{the solution of \eqref{e1} corresponding to $t=t^*$}, \ \ \ \mbox{strongly in $L^2(\Omega)$ as $j\rightarrow \infty$,}
		 	\end{equation}
		 	and this will be enough to imply the proposition. To see this, note that
		 	$m_j\otimes m_j\nabla p_j=(m_j\cdot\nabla p_j)m_j$, and thus we have the equation
		 	\begin{equation}\label{r333}
		 	-\mbox{div}(\nabla p_j+(m_j\cdot\nabla p_j)m_j)=S(x)\ \ \ \mbox{in $\Omega$.}
		 	\end{equation}
		 	Using $p_j$ as a test function, we can easily derive
		 	\begin{equation}\label{r222}
		 	\io|\nabla p_j|^2dx+\io(m_j\cdot\nabla p_j)^2dx\leq c\io |S(x)|^2dx.
		 	\end{equation}
		 	Thus we may assume that
		 	\begin{equation}
		 	p_j\rightharpoonup p \ \ \mbox{weakly in $W^{1,2}_0(\Omega)$ and strongly in $L^2(\Omega)$ }
		 	\end{equation}
		 	(passing to a subsequence if need be.)
		 	This together with (\ref{r111}) implies
		 	$$m_j\cdot\nabla p_j\rightharpoonup m^*\cdot \nabla p \ \ \mbox{weakly in $L^1(\Omega)$, and therefore also weakly in $L^2(\Omega)$.}$$
		 	Subsequently, we have 
		 	$$(m_j\cdot\nabla p_j)m_j\rightharpoonup ( m^*\cdot \nabla p)m^* \ \ \mbox{weakly in $\left(L^1(\Omega)\right)^N$.}$$ Thus we can take $j\rightarrow\infty$ in (\ref{r333}) to obtain
		 	\begin{equation}
		 	-\mbox{div}(\nabla p+(m^*\cdot\nabla p)m^*)=S(x)\ \ \ \mbox{in $\Omega$.}
		 	\end{equation}
		 	The solution to this equation is unique in $W^{1,2}_0(\Omega)$, and therefore $p=p^*$
		 	and the whole sequence $\{p_j\}$ tends to $p^*$ strongly in $L^2(\Omega)$. The proof is complete.
		 	\end{proof}	

		\section{Local estimates}
		
		In this section we begin with a derivation of local versions of (\ref{f1}) and (\ref{f3}). Then we proceed to prove Proposition 1.3. 
	
		Let $z=(y,\tau)\in\ot, r>0$ with $\qzr\subset\ot$ be given. Pick a $C^\infty$ function $\xi$ on $\mathbb{R}^{N+1}$ satisfying
		\begin{eqnarray*}
		\xi&=& 1\ \ \ \mbox{on $Q_{\frac{1}{2}r}(z)$},\\
		\xi&=& 0\ \ \ \mbox{off $Q_{r}(z)$},\\	
		0\leq&\xi&\leq 1\ \ \ \mbox{on $Q_{r}(z)$},\\
		|	\partial_t\xi|&\leq&\frac{c}{r^2},\\
		|	\nabla\xi|&\leq&\frac{c}{r}.
		\end{eqnarray*}
		Note that $m\otimes m\nabla p=\mnp m$. Keep this in mind, while using $\xi^2(p-\prt)$ as a test function in (\ref{e1}),  to obtain
		\begin{eqnarray}
		\lefteqn{\int_{B_r(y)}|\nabla p|^2\xi^2dx+\int_{B_r(y)}\mnp^2\xi^2dx}\nonumber\\
		&\leq&\frac{c}{r^2}\int_{B_r(y)}|p-\prt|^2dx+\frac{c}{r^2}\int_{B_r(y)}|m|^2|p-\prt|^2dx\nonumber\\
		&&+\int_{B_r(y)} |S(x)|\xi^2|p-\prt|dx.\label{pe1}
		\end{eqnarray}
		Set $M_0=\mbox{ess sup}_{\Omega_T}|p(x,t)|$. Then the fourth integral in (\ref{pe1})
		can be estimated as follows:
		\begin{eqnarray}
		\int_{B_r(y)}|m|^2|p-\prt|^2dx&\leq&2\int_{B_r(y)}|m-\mz|^2|p-\prt|^2dx\nonumber\\
		&&+2|\mz|^2\int_{B_r(y)}|p-\prt|^2dx\nonumber\\
		&\leq&8M_0^2\int_{B_r(y)}|m-\mz|^2dx\nonumber\\
		&&+2|\mz|^2\int_{B_r(y)}|p-\prt|^2dx.\label{pe2}
		\end{eqnarray}
		We apply Poincar\'{e}'s inequality to the last integral in (\ref{pe1}) to yield
		\begin{eqnarray}
		\int_{B_r(y)} |S(x)|\xi^2|p-\prt|dx&\leq&\left(\int_{B_r(y)}|S(x)|^{\frac{2N}{N+2}}\right)^{\frac{N+2}{2N}}\nonumber\\
		&&\cdot\left(\int_{B_r(y)}|\xi(p-\prt)|^{\frac{2N}{N-2}}\right)^{\frac{N-2}{2N}}\nonumber\\
		&\leq& \left(\int_{B_r(y)}|S(x)|^{\frac{2N}{N+2}}\right)^{\frac{N+2}{2N}}\nonumber\\
		&&\cdot\left(\frac{c}{r^2}\int_{B_r(y)}|p-\prt|^2dx+\int_{B_r(y)} \xi^2|\nabla p|^2dx\right)^{\frac{1}{2}}\nonumber\\
		&\leq&\varepsilon \int_{B_r(y)} \xi^2|\nabla p|^2dx+\frac{c\varepsilon}{r^2}\int_{B_r(y)}|p-\prt|^2dx\nonumber\\
		&&+c(\varepsilon)r^{N+2-\frac{2N}{q}}\label{pe3}
		\end{eqnarray}
		for each $\varepsilon>0$. Use (\ref{pe3}) and (\ref{pe2}) in (\ref{pe1}), choose
		$\varepsilon$ sufficiently small in the resulting inequality, and thereby arrive at
		\begin{eqnarray}
		\lefteqn{\int_{B_r(y)}|\nabla p|^2\xi^2dx+\int_{B_r(y)}\mnp^2\xi^2dx}\nonumber\\
		&\leq&\frac{c(1+|\mz|^2)}{r^2}\int_{B_r(y)}|p-\prt|^2dx+\frac{c}{r^2}\int_{B_r(y)}|m-\mz|^2dx\nonumber\\
		&&+cr^{N+2-\frac{2N}{q}}.\label{pe4}
		\end{eqnarray}
			Now we use $(m-\mz)\xi^2$ as a test function in (\ref{e2}) to obtain
		\begin{eqnarray}
		\lefteqn{\frac{d}{dt}\int_{B_r(y)}\frac{1}{2}|m-\mz|^2\xi^2dx+c\int_{B_r(y)} |\nabla m|^2\xi^2dx+\int_{B_r(y)}|m|^{2\gamma}\xi^2 dx}\nonumber\\
		&\leq& \frac{c}{r^2}\int_{B_r(y)}|m-\mz|^2dx+E^2\int_{B_r(y)}\mnp^2\xi^2dx\nonumber\\
		&&+c|\mz|^2\int_{B_r(y)}|\nabla p|^2\xi^2dx+\mz\int_{B_r(y)}|m|^{2(\gamma-1)}m\xi^2 dx.\label{me1}
		\end{eqnarray}
In view of the interpolation inequality (\cite{GT}, p. 145), we have
\begin{equation}
\left|\mz\int_{B_r(y)}|m|^{2(\gamma-1)}m\xi^2 dx\right|\leq \varepsilon\int_{B_r(y)}|m|^{2\gamma}\xi^2 dx+c(\varepsilon)|\mz|^{2\gamma}r^N, \ \ \ \varepsilon>0.\label{me2}
\end{equation}	
Substitute (\ref{me2})	into (\ref{me1}), choose $\varepsilon$ so small in the resulting inequality that the second integral in (\ref{me2}) can be absorbed into the third term in (\ref{me1}), then integrate
with respect to $t$ to yield
\begin{eqnarray}
\lefteqn{\maxth\ibh\frac{1}{2}|m-\mz|^2dx}\nonumber\\
&&+c\iqh |\nabla m|^2dxdt+\iqh|m|^{2\gamma} dxdt\nonumber\\
&\leq&c(|\mz|^2+1)\left(\int_{Q_{r}(z)}|\nabla p|^2\xi^2dxdt+\int_{Q_{r}(z)}\mnp^2\xi^2dxdt\right)\nonumber\\
&&+\frac{c}{r^2}\int_{Q_{r}(z)}|m-\mz|^2dxdt+c|\mz|^{2\gamma}r^{N+2}.\label{me3}
\end{eqnarray}

We are ready to prove Proposition \ref{prop1.3}. 

\begin{proof}[Proof of Proposition \ref{prop1.3}.] For each $\varepsilon>0$ we consider the set
\begin{equation}\label{me4}
H_\varepsilon=\{z\in\ot: 
\lim_{r\rightarrow 0}\frac{1}{r^{N+\varepsilon}}\int_{Q_{r}(z)}\left(|m|^{d}+|\partial_tm|^2+|\nabla m|^2+|\nabla p|^2+\mnp^2\right)dxdt=0 \},
\end{equation}
where $d=\frac{2N}{N-2}$ if $N\ne 2$ and any number bigger than $2+\frac{8}{N}$ if $N=2$.
On account of Lemma 1.1, we have	
\begin{equation}
\mathcal{P}^{N+\varepsilon}(\ot\setminus H_\varepsilon)=0.
\end{equation}
Thus it is enough for us to show
\begin{equation}
H_\varepsilon\subset R,\label{rw11}
\end{equation}
where $R$ is defined in (\ref{rt66}). We divide the proof of this  into several claims.\end{proof}
\begin{clm} If $z=(y,\tau)\in H_\varepsilon$, then we have
	\begin{equation}
	\sup_{r>0}|\mz|<\infty. 
	\end{equation}
\end{clm}
\begin{proof}
We follow the argument given in (\cite{G}, p. 104). That is, we calculate
\begin{eqnarray}
\left|\frac{d}{d\rho}m_{z,\rho}\right|&=&\left|\frac{d}{d\rho}\avint_{Q_1(0)}m(y+\zeta\rho,\tau+\rho^2\omega)d\zeta d\omega\right|\nonumber\\
&=&\left|\avint_{Q_1(0)}\left(\nabla m(y+\zeta\rho,\tau+\rho^2\omega)\zeta+\partial_\omega m(y+\zeta\rho,\tau+\rho^2\omega)2\rho\omega\right)d\zeta d\omega\right|\nonumber\\
&\leq&\aizr|\nabla m|dxdt+2\aizr|\partial_tm \rho|dxdt\nonumber\\
&\leq&c\left(\frac{1}{\rho^{N+2}}\int_{Q_\rho(z)}|\nabla m|^2dxdt\right)^{\frac{1}{2}}+c\left(\frac{1}{\rho^{N}}\int_{Q_\rho(z)}|\partial_tm|^2dxdt\right)^{\frac{1}{2}}\nonumber\\
&=&\frac{1}{\rho^{1-\frac{\varepsilon}{2}}}\left[\left(\frac{1}{\rho^{N+\varepsilon}}\int_{Q_\rho(z)}|\nabla m|^2dxdt\right)^{\frac{1}{2}}+c\left(\frac{1}{\rho^{N-2+\varepsilon}}\int_{Q_\rho(z)}|\partial_tm|^2dxdt\right)^{\frac{1}{2}}\right]\nonumber\\
&\leq&\frac{c}{\rho^{1-\frac{\varepsilon}{2}}}.
\end{eqnarray}
Here and in the remainder of the proof of Proposition 1.3 the constant $c$  may depend on  $\varepsilon$ and
$z$.
It immediately follows that
\begin{eqnarray}
\left|m_{z,\rho_1}-m_{z,\rho_2}\right|&\leq&\left|\int_{\rho_1}^{\rho_2}\left|\frac{d}{d\rho}m_{z,\rho}\right|d\rho\right|\nonumber\\
&\leq&c\left|\rho_1^{\frac{\varepsilon}{2}}-\rho_2^{\frac{\varepsilon}{2}}\right|.
\end{eqnarray}
Thus the claim follows.
\end{proof}
\begin{clm} If $z\in H_\varepsilon$, then
	\begin{equation}
	\avint_{Q_\rho(z)} |m-\mz|^2 dxdt\leq cr^\varepsilon.
	\end{equation}
	\end{clm}
\begin{proof}
	Note that 
	$$\mz=\frac{1}{r^2}\int_{\tau-\frac{1}{2}r^2}^{\tau+\frac{1}{2}r^2}\avint_{B_r(y)} m(x,t)dxdt=\frac{1}{r^2}\int_{\tau-\frac{1}{2}r^2}^{\tau+\frac{1}{2}r^2}m_{y,r}(t)dt.$$
	That is, $\mz$ is the average of $\mrt$ over $[\tau-\frac{1}{2}r^2, \tau+\frac{1}{2}r^2]$.
	Subsequently, we have
	\begin{eqnarray}
	|m_{y,r}(t)-\mz|&\leq&\int_{\tau-\frac{1}{2}r^2}^{\tau+\frac{1}{2}r^2}\left|\frac{d}{d\omega}m_{y,r}(\omega)\right|d\omega\nonumber\\
	&\leq&r\left(\int_{\tau-\frac{1}{2}r^2}^{\tau+\frac{1}{2}r^2}\avint_{B_r(y)}|\partial_\omega m|^2 dxd\omega \right)^{\frac{1}{2}},\label{rm12}
	\end{eqnarray}
	from whence follows
	\begin{equation}
	\frac{1}{r^2}\int_{\tau-\frac{1}{2}r^2}^{\tau+\frac{1}{2}r^2}|m_{y,r}(t)-\mz|^2 dt\leq \frac{c}{r^{N-2}}\int_{Q_{r}(z)} |\partial_tm|^2dxdt.
	\end{equation}
	In view of Poincar\'{e}'s inequality(\cite{EG}, p.141), we have
	\begin{equation}
	\avint_{B_r(y)}|m-m_{y,r}(t)|^2 dx\leq  cr^2\avint_{B_r(y)}|\nabla m|^2 dx.
	\end{equation}
	We compute
	\begin{eqnarray}
	\avint_{Q_\rho(z)} |m-\mz|^2 dxdt&=&\frac{1}{r^2}\int_{\tau-\frac{1}{2}r^2}^{\tau+\frac{1}{2}r^2}\avint_{B_r(y)}|m-\mz|^2 dxdt\nonumber\\
	&\leq&\frac{2}{r^2}\int_{\tau-\frac{1}{2}r^2}^{\tau+\frac{1}{2}r^2}\avint_{B_r(y)}|m-m_{y,r}(t)|^2 dxdt\nonumber\\
	&&+\frac{2}{r^2}\int_{\tau-\frac{1}{2}r^2}^{\tau+\frac{1}{2}r^2}|m_{y,r}(t)-\mz|^2 dt\nonumber\\
	&\leq&\frac{c}{r^N}\int_{Q_{r}(z)}|\nabla m|^2dxdt+\frac{c}{r^{N-2}}\int_{Q_{r}(z)}|\partial_tm|^2dxdt\nonumber\\
	&\leq&cr^\varepsilon.\label{rm13}
	\end{eqnarray}
	This completes the proof.
	\end{proof}
\begin{clm}Let $z\in H_\varepsilon$.
	Then for each $\alpha\in (0, \min\{\frac{2}{N-2}, \frac{4}{N}\}]$ there is a positive number $c$ such that
	\begin{equation}
\mat\avint_{B_r(y)}|m-\mrt|^{2+\alpha}dx\leq cr^{\frac{(2\alpha+2)\varepsilon}{d}}.
	\end{equation}
\end{clm}
\begin{proof} 
	 It follows from (\ref{rm13}) and (\ref{me3})  that
	\begin{equation}
	\mat\avint_{B_r(y)}|m-\mrt|^{2}dx\leq cr^\varepsilon.
	\end{equation}
	Note that the corollary  in (\cite{ST}, p.144) is not applicable here.  We offer a direct proof.
To this end,	we estimate from Poincar\'{e}'s inequality that
	\begin{eqnarray}
\lefteqn{	\frac{1}{r^2}\int_{\tau-\frac{1}{2}r^2}^{\tau+\frac{1}{2}r^2}\avint_{B_r(y)}|m-\mrt|^{\frac{4}{N}+2}dxdt}\nonumber\\
&\leq&\frac{1}{r^2}\int_{\tau-\frac{1}{2}r^2}^{\tau-\frac{1}{2}r^2}\left(\avint_{B_r(y)}|m-\mrt|^2dx\right)^{\frac{2}{N}}\left(\avint_{B_r(y)}|m-\mrt|^{\frac{2N}{N-2}}dx\right)^{\frac{N-2}{N}}dt\nonumber\\
&\leq&c\left(\mat\avint_{B_r(y)}|m-\mrt|^2dx\right)^{\frac{2}{N}}\nonumber\\
&&\cdot\int_{\tau-\frac{1}{2}r^2}^{\tau-\frac{1}{2}r^2}\avint_{B_r(y)}|\nabla m|^2dxdt\leq cr^{\varepsilon+\frac{2\varepsilon}{N}}.\label{rm21}
	\end{eqnarray}
	Let $\alpha\in (0,\min\{\frac{2}{N-2}, \frac{4}{N}\} ]$ be given. For $t\in[\tau-\frac{1}{2}r^2, \tau+\frac{1}{2}r^2]$ set
	$$f_r(t)=\avint_{B_r(y)}|m-\mrt|^{2+\alpha}dx.$$
Observe that
\begin{eqnarray}
|m-\mrt|^{2+2\alpha}&\leq&2^{2\alpha+1}\left(|m|^{2+2\alpha}+\avint_{B_r(y)}|m|^{2+2\alpha}dx\right),\\
\left|\frac{d}{dt}\mrt\right|^2&\leq&\left(\avint_{B_r(y)}|\partial_tm|dx\right)^2\leq \avint_{B_r(y)}|\partial_tm|^2dx.
\end{eqnarray}	
Keeping these two inequalities in mind, we calculate that
\begin{eqnarray}
\left|\frac{d}{dt}f_r(t)\right|&=&(2+\alpha)\left|\avint_{B_r(y)}|m-\mrt|^\alpha(m-\mrt)\cdot\frac{d}{dt}(m-\mrt)dt\right|\nonumber\\
&\leq&c\avint_{B_r(y)}|m-\mrt|^{\alpha+1}|\partial_tm-\frac{d}{dt}\mrt|dx\nonumber\\
&\leq&c\left(\avint_{B_r(y)}|m-\mrt|^{2\alpha+2}dx\right)^{\frac{1}{2}}\left(\avint_{B_r(y)}|\partial_tm-\frac{d}{dt}\mrt|^2dx\right)^{\frac{1}{2}}\nonumber\\
&\leq&c\left(\avint_{B_r(y)}|m|^{2\alpha+2}dx\right)^{\frac{1}{2}}\left(\avint_{B_r(y)}|\partial_tm|^2dx\right)^{\frac{1}{2}}.
\end{eqnarray}
Note that $2+2\alpha\leq d$, where $d$ is given in \eqref{me4}. We estimate
\begin{eqnarray}
\mat f_r(t)&\leq &\mat\left|f_r(t)-\frac{1}{r^2}\int_{\tau-\frac{1}{2}r^2}^{\tau+\frac{1}{2}r^2}f_r(\omega)d\omega\right|\nonumber\\
&&+\frac{1}{r^2}\int_{\tau-\frac{1}{2}r^2}^{\tau+\frac{1}{2}r^2}f_r(\omega)d\omega\nonumber\\
&\leq&\int_{\tau-\frac{1}{2}r^2}^{\tau+\frac{1}{2}r^2}\left|\frac{d}{dt}f_r(t)\right|dt
+\frac{1}{r^2}\int_{\tau-\frac{1}{2}r^2}^{\tau+\frac{1}{2}r^2}f_r(t)dt\nonumber\\
&\leq&c\left(\frac{1}{r^N}\int_{Q_{r}(z)}|m|^{2\alpha+2}dx\right)^{\frac{1}{2}}\left(\frac{1}{r^N}\int_{Q_{r}(z)}|\partial_tm|^2dx\right)^{\frac{1}{2}}\nonumber\\
&&+\frac{1}{r^2}\int_{\tau-\frac{1}{2}r^2}^{\tau+\frac{1}{2}r^2}\avint_{B_r(y)}|m-\mrt|^{\alpha+2}dxdt\nonumber\\
&\leq&c\left(\frac{1}{r^N}\int_{Q_{r}(z)}|m|^{d}dx\right)^{\frac{2\alpha+2}{2d}}\left(\frac{1}{r^N}\int_{Q_{r}(z)}|\partial_tm|^2dx\right)^{\frac{1}{2}}\nonumber\\
&&+\left(\frac{1}{r^2}\int_{\tau-\frac{1}{2}r^2}^{\tau+\frac{1}{2}r^2}\avint_{B_r(y)}|m-\mrt|^{\frac{4}{N}+2}dxdt\right)^{\frac{2+\alpha}{2+\frac{4}{N}}}\nonumber\\
&\leq&cr^{\frac{(2\alpha+2)\varepsilon}{d}}+
cr^{\varepsilon(1+\frac{\alpha}{2})}\leq cr^{\frac{(2\alpha+2)\varepsilon}{d}}.
\end{eqnarray}
The proof is complete.
	\end{proof}
	\begin{clm}
		If $z\in H_\varepsilon$, then there is $\varepsilon_1>0$ such that
		\begin{equation}
	A_r(z)\leq cr^{\varepsilon_1}.
		\end{equation}
	\end{clm}
	Obviously, this claim implies (\ref{rw11}).
	\begin{proof}
Let $z=(y,\tau)\in H_\varepsilon$ be given. Fix $r>0$ with
$\qzr\subset\ot$. 
Set
\begin{equation}
w_r=m-\mrt.
\end{equation}
Note that
\begin{eqnarray*}
m\otimes m&=&(m-\mrt)\otimes m+\mrt\otimes(m-\mrt)\\
&&+\mrt\otimes\mrt\nonumber\\
&=&w_r\otimes m+\mrt\otimes w_r+\mrt\otimes\mrt.
\end{eqnarray*}
Thus $p$ satisfies the system
\begin{eqnarray}
\lefteqn{-\mbox{div}[(I+\mrt\otimes\mrt)\nabla p]}\nonumber\\
&&=\mbox{div}\left[(m\cdot\nabla p)w_r\right]+\mbox{div}\left[(w_r\cdot\nabla p)\mrt\right]+S(x)\ \ \ \mbox{in $\qzr$.}\label{bl00}
\end{eqnarray}
Here we have used the fact that $(w_r\otimes m)\nabla p=\mnp w_r$.
We decompose $p$ into $\eta+\phi $ on $\qzr$ as follows: $\eta$ is the solution of the problem
\begin{eqnarray}
-\mbox{div}\left[(I+\mrt\otimes\mrt)\nabla\eta\right]&=& 0\nonumber\\
&& \ \ \mbox{in $\byr$,}\ \ t\in [\tau-\frac{1}{2}r^2, \tau+\frac{1}{2}r^2],\label{eta11}\\
\eta&=&p\ \ \ \mbox{on $\partial\byr$},\ \ t\in [\tau-\frac{1}{2}r^2, \tau+\frac{1}{2}r^2],\label{eta22}
\end{eqnarray}
while $\phi$ is the solution of the problem
\begin{eqnarray}
-\mbox{div}\left[(I+\mrt\otimes\mrt)\nabla\phi\right]&=& \mbox{div}\left[(m\cdot\nabla p)w_r\right]+\mbox{div}\left[(w_r\cdot\nabla p)\mrt\right]\nonumber\\
&&+S(x) \ \ \mbox{in $\byr$,}\ \ t\in [\tau-\frac{1}{2}r^2, \tau+\frac{1}{2}r^2],\label{p11}\\
\phi&=&0\ \ \ \mbox{on $\partial\byr$,}\ \ t\in [\tau-\frac{1}{2}r^2], \tau+\frac{1}{2}r^2\label{p22}
\end{eqnarray}
Recall from (\ref{rm12}) that
\begin{eqnarray}
|\mrt|&\leq&|\mrt-\mz|+|\mz|\nonumber\\
&\leq&cr\left(\frac{1}{r^N}\int_{Q_{r}(z)}|\partial_t m|^2dxdt\right)^{\frac{1}{2}}+|\mz|.
\end{eqnarray}
By Theorem 2.1 in (\cite{G}, p.78), there is a positive number $c$ depending only on $\sup_{r>0}|\mrt|$ such that
\begin{equation}
\air|\eta-\eta_{y,\rho}|^2dx\leq c\left(\frac{\rho}{R}\right)^2\aiR|\eta-\eta_{y,R}|^2dx\label{rw12}
\end{equation}
for all $0<\rho\leq R\leq r$ and $t\in [\tau-\frac{1}{2}r^2, \tau+\frac{1}{2}r^2]$. On the other hand, another classical regularity result \cite{SI} for linear elliptic equations with continuous coefficients asserts that for each $s\in (1,\infty)$ there is a positive number
$c$ with the property
\begin{eqnarray}
\|\nabla\phi\|_{s}&\leq& c\|(m\cdot\nabla p) w_r\|_{s}+c\|(w_r\cdot\nabla p)\mrt\|_{s}\nonumber\\
&&+c\|S(x)\|_{\frac{sN}{s+N}},\ \ t\in [\tau-\frac{1}{2}r^2, \tau+\frac{1}{2}r^2].
\end{eqnarray}
Note that the constant $c$ here is also independent of $r$.
We remark that in general the above inequality is not true for $s=1$. This is why Claim 3.3 is crucial to our development.
Obviously, if we replace $\mz$ by
 $\mrt$ in (\ref{pe4}), the resulting inequality still holds. This implies
\begin{equation}
\mat\frac{1}{r^{N-2}}\int_{B_r(y)}\left(|\nabla p|^2+\mnp^2\right)dx\leq c.
\end{equation}
We can easily find a $s\in(1,2)$ so that
\begin{equation}
\frac{2s}{2-s}=2+\frac{4(s-1)}{2-s}\leq2+\min\{\frac{2}{N-2},\frac{4}{N}\}.
\end{equation}
We estimate
\begin{eqnarray}
\lefteqn{\mat\frac{1}{r^{N-s}}\int_{B_r(y)} |\mnp w_r|^sdx}\nonumber\\
&\leq& \left(\mat\frac{1}{r^{N-2}}\int_{B_r(y)} \mnp^2dx\right)^{\frac{s}{2}}\nonumber\\
&&\cdot \left(\mat\frac{1}{r^{N}}\int_{B_r(y)}|w_r|^\frac{2s}{2-s}dx \right)^{\frac{2-s}{2}}\nonumber\\
&\leq& c\left(\mat\frac{1}{r^{N}}\int_{B_r(y)}|w_r|^\frac{2s}{2-s}dx \right)^{\frac{2-s}{2}}\nonumber\\
&\leq& cr^{\frac{(2-s)(2\alpha+2)\varepsilon}{2d}},
\end{eqnarray}
where $\alpha=\frac{4(s-1)}{2-s}$.
Similarly, we have
\begin{eqnarray}
\lefteqn{\mat\frac{1}{r^{N-s}}\int_{B_r(y)} | (w_r\cdot\nabla p)\mrt|^sdx}\nonumber\\
&\leq& \left(\mat\frac{1}{r^{N-2}}\int_{B_r(y)}|\nabla p|^2dx\right)^{\frac{s}{2}}\nonumber\\
&&\cdot \left(\mat\frac{1}{r^{N}}\int_{B_r(y)}|w_r|^\frac{2s}{2-s}dx \right)^{\frac{2-s}{2}}\nonumber\\
&\leq& c\left(\mat\frac{1}{r^{N}}\int_{B_r(y)}|w_r|^\frac{2s}{2-s}dx \right)^{\frac{2-s}{2}}\nonumber\\
&\leq& cr^{\frac{(2-s)(2\alpha+2)\varepsilon}{2d}},\\
\lefteqn{\frac{1}{r^{N-s}}\left(\int_{B_r(y)}|S(x)|^{\frac{Ns}{N+s}}dx\right)^{\frac{s+N}{N}}}\nonumber\\
&\leq& \frac{1}{r^{N-s}}\left(\int_{B_r(y)}|S(x)|^{q}dx\right)^{\frac{s}{q}}r^{s+N-\frac{Ns}{q}}
\leq cr^{s(2-\frac{N}{q})}.
\end{eqnarray}
To summarize, we have
\begin{equation}
\mat\frac{1}{r^{N-s}}\int_{B_r(y)} |\nabla\phi|^sdx\leq cr^{\min\{{\frac{(2-s)(2\alpha+2)\varepsilon}{2d}},\ s(2-\frac{N}{q})\}}.
\end{equation}
It follows from Poincar\'{e}'s inequality that
\begin{equation}
\left(\avint_{B_r(y)}|\phi-\phi_{y,r}(t)|^{\frac{Ns}{N-s}}dx\right)^{\frac{N-s}{Ns}}\leq cr\left(\avint_{B_r(y)}|\nabla\phi|^sdx\right)^{\frac{1}{s}}=c\left(\frac{1}{r^{N-s}}\int_{B_r(y)} |\nabla\phi|^sdx\right)^{\frac{1}{s}}.
\end{equation}
Remember that $\|\phi\|_\infty\leq \|\eta\|_\infty+\|p\|_\infty\leq 2\|p\|_\infty$.
Hence we can always find a positive number $\varepsilon_1\in(0,2)$ so that
\begin{equation}
\mat\avint_{B_r(y)}|\phi-\phi_{y,r}(t)|^2dx\leq cr^{\varepsilon_1}.\label{rw13}
\end{equation}
For $0<\rho\leq r$ we derive from (\ref{rw12}) and (\ref{rw13}) that
\begin{eqnarray}
\lefteqn{ \int_{B_\rho(y)}|p-p_{y,\rho}(t)|^2dx}\nonumber\\
&\leq &2 \int_{B_\rho(y)}|\eta-\eta_{y,\rho}(t)|^2dx+2 \int_{B_\rho(y)}|\phi-\phi_{y,\rho}(t)|^2dx\nonumber\\
&\leq&c\left(\frac{\rho}{r}\right)^{N+2}\int_{B_r(y)}|\eta-\eta_{y,r}(t)|^2dx+2 \int_{B_r(y)}|\phi-\phi_{y,r}(t)|^2dx\nonumber\\
&\leq&c\left(\frac{\rho}{r}\right)^{N+2}\int_{B_r(y)}|p-p_{y,r}(t)|^2dx+cr^{N+\varepsilon_1}.\label{rw14}
\end{eqnarray}
Here we have used the fact that $\int_{B_\rho(y)}|\phi-\phi_{y,\rho}(t)|^2dx$ is an increasing function of $\rho$. We set
$$
\sigma(r)=\mat\int_{B_r(y)}|p-p_{y,r}(t)|^2dx.
$$
We easily infer from (\ref{rw14}) that 
\begin{equation}
\sigma(\rho)\leq c\left(\frac{\rho}{r}\right)^{N+2}\sigma(r)+cr^{N+\varepsilon_1}.
\end{equation}
for all $0<\rho\leq r$. This puts us in a position to apply Lemma 2.1 in (\cite{G}, p.86), from whence follows
\begin{equation}
\sigma(\rho)\leq c\left(\frac{\rho}{r}\right)^{N+\varepsilon_1}\sigma(r)+c\rho^{N+\varepsilon_1}
\end{equation}
for all $0<\rho\leq r$.
This gives the claim.
\end{proof}

		\section{Proof of Proposition \ref{pro12}}	
		In this section we present the proof of Proposition \ref{pro12}. We would like to remark that the proof of this proposition is more challenging than that of Proposition \ref{prop1.3} mainly because we do not have a local estimate for
		$\partial_t m$ or a local $L^\infty$ estimate for $p$. This also causes us to impose the restriction $N\leq 3$. Note that this restriction is not needed in Propositions \ref{prop1.1} and \ref{prop1.3}.
			
			\begin{proof}[Proof of Proposition \ref{pro12}.]
			We argue by contradiction. Suppose that the proposition is false. Then for some $M>0$ (\ref{con1}) and (\ref{con2}) fail to hold no matter how we pick numbers $\varepsilon, \delta$ from the interval $(0,1)$. In particular, we can choose a sequence $\{\varepsilon_k\} \subset (0,1)$ with the property
			\begin{equation}
			\varepsilon_k\rightarrow 0 \ \ \ \mbox{as $k\rightarrow 0$.}\label{con3}
			\end{equation}
			The selection of $\delta$ from  $(0,1)$ is more delicate, and it will be made clear later.
	Let $\delta$ be chosen as below. For each $k$ there exist cylinders $\qzrk\subset\ot$ such that
			\begin{equation}
			|\mzk|\leq M\ \ \mbox{and}\ \ \ezk\leq\varepsilon_k,
			\end{equation}		
	whereas				
					\begin{equation}
					E_{\delta r_k}(z_k)> \frac{1}{2}\ezk, \ \ \ k=1,\cdots.
					\end{equation}	
				Set
					$$
					\lambda^2_k=\ezk.
					$$
					Then (\ref{con3}) asserts
					$$\lambda_k\rightarrow 0\ \ \ \mbox{as $k\rightarrow\infty$.}$$
					We rescale our variables to the unit cylinder $Q_1(0)$, as follows. If $z=(y,\tau)\in Q_1(0)$, write
					\begin{eqnarray}
					\psi_k(y,\tau)&=&\frac{p(\ykr,\tkr)-\pkt}{\lambda_k},\\
					n_k(y,\tau)&=&m(\ykr,\tkr),\\
					w_k(y,\tau)&=&\frac{n_k(y,\tau)-\mzk}{\lambda_k}.
					\end{eqnarray}
					We can easily verify
\begin{eqnarray*}
\maxn\int_{B_1(0)}\psi_k^2(y,\tau)dy
&=&\frac{1}{\lambda_k^2}A_{r_k}(z_k)\leq 1,\\
\int_{Q_1(0)}|w_k(y,\tau)|^2dyd\tau&=&\frac{1}{\lambda_k^2r_k^{N+2}}\iqk|m(x,t)-\mzk|^2dxdt\leq 1,
\end{eqnarray*}
but
\begin{eqnarray}
\frac{1}{\delta^{N+2}}\int_{Q_{\delta}(0)}|w_k-(w_k)_{0,\delta}|^2dyd\tau
	&+&\frac{1}{\delta^{N}}\maxd\int_{B_{\delta}(0)}|\psi_k-(\psi_k)_{0,\delta}(\tau)|^2dy\nonumber\\
	&&+
		\frac{\delta^{2\beta}r_k^{2\beta}}{\lambda_k^2}>\frac{1}{2}.\label{con}
\end{eqnarray}
Here and in what follows we suppress the dependence of $\sk,w_k, n_k$ on $(y,\tau)$ for simplicity of notation.
Our plan is to show that the lim sup of the left-hand side of the above inequality as $k\rightarrow\infty$ can be made smaller than $\frac{1}{2}$ if we adjust $\delta$ to be small enough, and thus the desired contradiction follows. 

			We easily see from the definition of $\lambda_k$ that 
			\begin{equation}
			\frac{\delta^{2\beta}r_k^{2\beta}}{\lambda_k^2}\leq\delta^{2\beta}.
			\end{equation}
To analyze the first two terms in (\ref{con}), we first conclude from the proof in \cite{E} that $\psi_k(y,\tau), w_k(y,\tau)$ satisfy the system
		\begin{eqnarray}
		&&-\Delta\psi_k-\mbox{div}\left[(n_k\cdot\nabla\psi_k)n_k\right]=\frac{r_k^2}{\lambda_k}S(\ykr)\equiv F_k(y)\ \ \ \mbox{in $Q_1(0)$,}\label{bl0}\\
		&&\partial_tw_k-D^2\Delta w_k-E^2\lambda_k(n_k\cdot\nabla\psi_k)\nabla\psi_k+\frac{r_k^2}{\lambda_k}|n_k|^{2(\gamma-1)}n_k=0\ \ \ \mbox{in $Q_1(0)$.}\label{bl1}
		\end{eqnarray}
		We can infer from (\ref{pe4}) that
		\begin{equation}
		\int_{Q_{\frac{1}{2}}(0)}|\nabla\psi_k|^2dyd\tau+	\int_{Q_{\frac{1}{2}}(0)}|n_k\cdot\nabla\psi_k|^2dyd\tau\leq c.\label{bl2}
		\end{equation}
	Similarly, we can derive from (\ref{me3}) that
		\begin{equation}
		\max_{\tau\in[-\frac{1}{8},\frac{1}{8}]}\int_{B_{\frac{1}{2}}(0)}|w_k|^2dy
			+\int_{Q_{\frac{1}{2}}(0)}|\nabla w_k|^2 dyd\tau+\frac{r_k^2}{\lambda_k^2}\int_{Q_{\frac{1}{2}}(0)}|n_k|^{2\gamma}dyd\tau
			\leq c+c\frac{r_k^2}{\lambda_k^2}\leq c.\label{bl3}
		\end{equation}
	Consequently, we have
	\begin{eqnarray}
		\int_{Q_{\frac{1}{2}}(0)}\left|\frac{r_k^2}{\lambda_k}|n_k|^{2\gamma-1}\right|^{\frac{2\gamma}{2\gamma-1}}dyd\tau&=&\lambda_k^{\frac{2\gamma}{2\gamma-1}}\left(\frac{r_k^2}{\lambda_k^2}\right)^{\frac{1}{2\gamma-1}}\frac{r_k^2}{\lambda_k^2}\int_{Q_{\frac{1}{2}}(0)}|n_k|^{2\gamma}dyd\tau\nonumber\\
		&&\rightarrow 0\ \ \mbox{as $k\rightarrow 0$.}\label{bl4}
	\end{eqnarray}	
	This together with (\ref{bl1}), (\ref{bl2}), and (\ref{bl3}) implies that the sequence $\{\partial_\tau w_k\}$ is bounded in $L^2(-\frac{1}{8}, \frac{1}{8}; W^{-1,2}(B_{\frac{1}{2}}(0)))+L^{1}(Q_{\frac{1}{2}}(0))$.
	 By a well-known result in \cite{S}, $w_k$ is precompact in $L^2(Q_{\frac{1}{2}}(0))$. Passing to subsequences if necessary, we have
	\begin{eqnarray}
	\mzk&\rightarrow& a,\\
	n_k=\lambda_k w_k+\mzk&\rightarrow& a\ \ \ \mbox{strongly in $L^2(Q_{1}(0))$,}\label{nsc}\\
w_k&\rightarrow& w \ \ \mbox{strongly in $L^2(Q_{\frac{1}{2}}(0))$}\label{ws}\nonumber\\
&&\mbox{ and weakly in $L^2(-\frac{1}{8},\frac{1}{8} ; W^{1,2}(B_{\frac{1}{2}}(0)))$,}\\
\psi_k&\rightarrow&\psi \ \ \ \mbox{ and weakly in $L^2(-\frac{1}{8},\frac{1}{8} ; W^{1,2}(B_{\frac{1}{2}}(0)))$.}\label{pwc}
	\end{eqnarray}
	In view of (\ref{bl2}) and (\ref{bl4}), we can send $k$ to infinity in (\ref{bl1}) to obtain
	\begin{equation}
	\partial_\tau w-D^2\Delta w=0 \ \ \mbox{in $Q_{\frac{1}{2}}(0)$}
	\end{equation}
in the weak, and therefore classical sense. It follows from (\ref{nsc}) and (\ref{pwc}) that 
\begin{equation}
n_k\nabla\psi_k\rightharpoonup a\nabla\psi \ \ \ \mbox{weakly in $L^1(Q_{\frac{1}{2}}(0))$,}
\end{equation}
and therefore weakly in $L^2(Q_{\frac{1}{2}}(0))$ due to (\ref{bl2}). This, in turns, implies
\begin{equation}
(n_k\nabla\psi_k)n_k\rightharpoonup a\nabla\psi a\ \ \ \mbox{weakly in $L^1(Q_{\frac{1}{2}}(0))$.}
\end{equation}
We estimate the last term in (\ref{bl0}) as follows
\begin{eqnarray}
\ibz|F_k|^qdy&=&\frac{r_k^{2q}}{\lambda_k^q}\ibz|S(\ykr)|^qdy\nonumber\\
&=&\frac{r_k^{2q-N}}{\lambda_k^q}\ibk|S(x)|^qdx\nonumber\\
&\leq& c\frac{r_k^{\beta q}}{\lambda_k^q}r_k^{q(2-\frac{N}{q}-\beta)}\leq cr_k^{q(2-\frac{N}{q}-\beta)}\rightarrow 0.\label{jf4}
\end{eqnarray}
The last step is due to (\ref{cb}). We are ready to let $k$ go to infinity in (\ref{bl0}), thereby obtaining
\begin{equation}
-\mbox{div}\left[(I+a\otimes a)\nabla\psi\right]= 0 \ \ \
\mbox{in $Q_{\frac{1}{2}}(0)$.}
\end{equation}
Remember that $a$ is a constant vector. By the classical regularity theory for
linear elliptic equations, there exist $c>0, \alpha\in (0,1)$ determined only by $M$ and $N$ with the property
\begin{equation}
\max_{\tau\in[-\frac{1}{2}]\delta^2,\frac{1}{2}\delta^2]}\avint_{B_\delta(0)}|\psi-\psi_{0,\delta}(\tau)|^2dy\leq\max_{\tau\in[-\frac{1}{2}]\delta^2,\frac{1}{2}\delta^2]} c\delta^{2\alpha}\avint_{B_{\frac{1}{2}}(0)}|\psi-\psi_{0,\frac{1}{2}}(\tau)|^2dy\leq c\delta^{2\alpha}
\end{equation} 
for all $\delta\leq \frac{1}{4}$.
Subsequently,
\begin{equation}
\frac{1}{\delta^{N}}\max_{\tau\in[-\frac{1}{2}]\delta^2,\frac{1}{2}\delta^2]}\int_{B_\delta(0)}|\psi-\psi_{0,\delta}(\tau)|^2dy\leq c\delta^{2\alpha}. 
\end{equation} 
for all $0<\delta\leq\frac{1}{4}$.
 It is also well-known (see, e.g., Claim 1 in \cite{X2}) that there exist $c>0, \alpha\in (0,1)$ determined only by $N, D$ such that
\begin{eqnarray}
\avint_{Q_\delta(0)}|w-w_{0,\delta}|^2dyd\tau\leq c\delta^{2\alpha}\avint_{Q_{\frac{1}{2}}(0)}|w-w_{0,\frac{1}{2}}|^2dyd\tau
\leq c\delta^{2\alpha}
\end{eqnarray}
for all $0<\delta\leq\frac{1}{4}$.

If we could pass to the limit in (\ref{con}), this would result in the desired contradiction. What prevents us
from doing so is the lack of compactness of the sequence $\{\psi_k\}$ in the t-variable. To circumvent this problem, we fix a suitably small number $\frac{1}{16}\geq\delta_0>0$ and consider the decomposition $\psi_k=\eta_k+\phi_k$ on $Q_{\delta_0}(0)$, where $\eta_k$ is the solution of the problem
\begin{eqnarray}
-\mbox{div}\left[(I+\mzk\otimes\mzk)\nabla\eta_k\right]&=& 0\ \ \ \mbox{in $\bdz$,}\ \ \tau\in [-\frac{1}{2}\delta^2_0, \frac{1}{2}\delta^2_0],\label{eta1}\\
\eta_k&=&\psi_k\ \ \ \mbox{on $\partial\bdz$},\ \ \tau\in [-\frac{1}{2}\delta^2_0, \frac{1}{2}\delta^2_0],\label{eta2}
\end{eqnarray}
while $\phi_k$ is the solution of the problem
\begin{eqnarray}
-\mbox{div}\left[(I+\mzk\otimes\mzk)\nabla\phi_k\right]&=& \lambda_k\mbox{div}((n_k\cdot\nabla\psi_k) w_k)+\lambda_k\mbox{div}((w_k\cdot\nabla\psi_k) \mzk)\nonumber\\
&&+F_k\ \ \ \mbox{in $\bdz$,}\ \ \tau\in [-\frac{1}{2}\delta^2_0, \frac{1}{2}\delta^2_0],\label{p1}\\
\phi_k&=&0\ \ \ \mbox{on $\partial\bdz$,}\ \ \tau\in [-\frac{1}{2}\delta^2_0, \frac{1}{2}\delta^2_0].\label{p2}
\end{eqnarray}
We will show that $\{\phi_k\}$ is precompact in $L^\infty(-\frac{1}{2}\delta_0^2,\frac{1}{2}\delta_0^2; L^2(B_{\delta_0}(0)) )$, and this is enough for our purpose in spite of the fact that $\{\eta_k\}$ may not be precompact in the preceding function space. To see this, we first infer from (\ref{pe4}) that
\begin{equation}
\max_{\tau\in[-\frac{1}{32},\frac{1}{32}]}\left(\int_{B_{\frac{1}{4}}(0)}|\nabla\psi_k|^2dy+\int_{B_{\frac{1}{4}}(0)}(n_k\cdot\nabla\psi_k)^2 dy\right)\leq c+
\max_{\tau\in[-\frac{1}{8},\frac{1}{8}]}\int_{B_{\frac{1}{2}}(0)}|w_k|^2dy\leq c.\label{psu}
\end{equation}
Using $\eta_k-\psi_k$ as a test function in (\ref{eta1}) yields
\begin{equation}
\max_{\tau\in [-\frac{1}{2}\delta^2_0, \frac{1}{2}\delta^2_0]}\ibd |\nabla\eta_k|^2 dy\leq \max_{\tau\in [-\frac{1}{2}\delta^2_0, \frac{1}{2}\delta^2_0]}c\ibd |\nabla\psi_k|^2 dy\leq c.\label{eb1}
\end{equation}
Note that (\ref{eta1}) is an uniformly elliptic equation with constant coefficients. The
classical regularity theory asserts that there exist $c>0, \alpha\in (0, 1)$ depending
only on $M, N$ such that
\begin{eqnarray}
\frac{1}{\delta^{N}}\int_{B_\delta(0)}|\eta_k-(\eta_k)_{0,\delta}(\tau)|^2dy&\leq& c\delta^{2\alpha}\avint_{B_{\delta_0}(0)}|\eta_k-(\eta_k)_{0,\delta_0}(\tau)|^2dy\nonumber\\
&\leq& c\delta^{2\alpha}\delta_0^2\avint_{B_{\delta_0}(0)}|\nabla\eta_k|^2dy\leq c\delta^{2\alpha}\label{eb2}
\end{eqnarray} 
for all $\delta\leq \frac{1}{2}\delta_0$.

Now we turn our attention to the sequence $\{\phi_k\}$. We wish to show 
\begin{equation}
\phi_k\rightarrow 0\ \ \ \mbox{strongly in $L^\infty(-\frac{1}{2}\delta_0^2, \frac{1}{2}\delta_0^2; L^2(B_{\delta_0}(0)) )$. }
\end{equation}  This is where the subtlety of our analysis lies. We observe from (\ref{psu}) that 
\begin{equation}\label{eb3}
\max_{\tau\in[-\frac{1}{32},\frac{1}{32}]}\int_{B_{\frac{1}{4}}(0)}|\psi_k|^{\frac{2N}{N-2}}dy\leq c.
\end{equation}
In view of (\ref{eb1}), $\{\phi_k\}$ also satisfies the above estimate. By the interpolation inequality (\cite{G}, p.146) 
\begin{equation}
\|\phi_k(\cdot, \tau)\|_2\leq \varepsilon\|\phi_k(\cdot, \tau)\|_{\frac{2N}{N-2}}+c(\varepsilon)\|\phi_k(\cdot, \tau)\|_1, \ \ \ \varepsilon>0,
\end{equation}
it is sufficient for us to show
\begin{equation}
\max_{\tau \in[-\frac{1}{2}\delta_0^2,\frac{1}{2}\delta_0^2]}\ibd|\phi_k(y,\tau)|dy\rightarrow 0\ \ \mbox{as $k\rightarrow \infty$.}
\end{equation}
Note that the elliptic coefficients in (\ref{p1}) are constants. This puts us in a position to invoke the classical $W^{1,s}$ estimate for $\phi_k$. That is, for each
$s\in (1, \infty)$ there is a positive number $c$ with the property
\begin{equation}
\|\nabla\phi_k\|_s\leq c\lambda_k\|(n_k\cdot\nabla\psi_k) w_k\|_s
+c\lambda_k\|((w_k\cdot\nabla\psi_k) \mzk\|_s+c\|F_k\|_{\frac{sN}{s+N}}.\label{pz}
\end{equation}
Remember that
 (\ref{pz}) does not hold for $s=1$. To find a $s>1$, we will show that there is a $\beta>0$ such that
\begin{equation}
\max_{\tau \in[-\frac{1}{2}\delta_0^2,\frac{1}{2}\delta_0^2]}\ibd|w_k(y,\tau)|^{2(1+\beta)}dy\leq c.\label{pb}
\end{equation}
Obviously, this will imply that 
\begin{equation}
\max_{\tau \in[-\frac{1}{2}\delta_0^2,\frac{1}{2}\delta_0^2]}\left(\|(n_k\cdot\nabla\psi_k) w_k\|_s
+\|(w_k\cdot\nabla\psi_k) \mzk\|_s\right)\leq c
\end{equation}
for some $s>1$. Consequently, the right-hand side of (\ref{pz}) goes to $0$ as $k\rightarrow\infty$. 
To establish (\ref{pb}), we will develop a suitable local version of (C2) in Proposition \ref{l1}. This effort is complicated by the fact that a local version of (C1) in the proposition is not available. The remaining part of this section will be dedicated to the proof of (\ref{pb}), which will be divided into two claims.

 \begin{clm} We have:
 	\begin{equation}
 	\int_{Q_{\frac{1}{8}}(0)}|\sk\nabla\sk|^2 dyd\tau+\int_{Q_{\frac{1}{8}}(0)}\nsk^2|\sk|^2 dyd\tau\leq c.
 	\end{equation}
 \end{clm}
 \begin{proof}
 	Let $\xi$ be a $C^\infty$ function on $\mathbb{R}^N\times\mathbb{R}$ with the properties
 	\begin{eqnarray}
 	\xi &=& 0\ \ \ \mbox{outside $Q_1(0)$, and}\label{rt11}\\
 	\xi&\in&[0,1]\ \ \ \mbox{in $Q_1(0)$.}\label{rt22}
 	\end{eqnarray}
 	Upon using $\sk^3\xi^2$ as a test function in (\ref{bl0}), we deduce
 	\begin{eqnarray}
 	\lefteqn{	\int_{B_1(0)}|\sk\nabla\sk|^2\xi^2 dy+\int_{B_1(0)}\nsk^2\sk^2\xi^2 dy}\nonumber\\
 	&\leq&c\int_{B_1(0)}\sk^4|\nabla\xi|^2 dy+c\int_{B_1(0)}|n_k|^2\sk^4|\nabla\xi|^2 dy+\int_{B_1(0)}|F_k||\sk|^3\xi^2 dy.\label{jf1}
 	\end{eqnarray}
 	Observe that 
 	\begin{equation}
 	|\lambda_k\sk|\leq c.
 	\end{equation}
 	Subsequently, we have
 	\begin{eqnarray}
 	|n_k|^2\sk^4&=&|\lambda_kw_k+\mzk|^2\sk^4\nonumber\\
 	&\leq& 2\lambda_k^2\sk^4|w_k|^2+c\sk^4\nonumber\\
 	&\leq& c\sk^2|w_k|^2+c\sk^4\nonumber\\
 	&\leq& c\psi_k^{\frac{2N}{N-2}}+c|w_k|^N+c\sk^4.\label{jf2}
 	\end{eqnarray}
 		We estimate from (\ref{bl3}) and the Sobolev Embedding Theorem that
 		\begin{eqnarray}
 		\int_{Q_{\frac{1}{2}}(0)}| w_k|^{2 +\frac{4}{N}}dyd\tau&\leq&\int_{-\frac{1}{8}}^{\frac{1}{8}}\left(\int_{B_{\frac{1}{2}}(0)}|w_k|^2dy\right)^{\frac{2}{N}}\left(\int_{B_{\frac{1}{2}}(0)}|w_k|^{\frac{2N}{N-2}}dy\right)^{\frac{N-2}{N}}d\tau\nonumber\\
 		&\leq&c\left(\max_{\tau\in(-\frac{1}{8},\frac{1}{8})}\int_{B_{\frac{1}{2}}(0)}|w_k|^2dy\right)^{\frac{2}{N}}\nonumber\\
 		&&\cdot
 		\left(\int_{Q_{\frac{1}{2}}(0)}|\nabla w_k|^2dyd\tau+\int_{Q_{\frac{1}{2}}(0)}| w_k|^2 dyd\tau\right)\nonumber\\
 		&\leq& c.\label{rt33}
 		\end{eqnarray}
 		Our assumption on the space dimension $N$ implies
 		$$N\leq 2+\frac{4}{N},\ \ \ \frac{2N}{N-2}>4.$$
 		By virtue of (\ref{psu}), we obtain	\begin{eqnarray}
 		\int_{B_{\frac{1}{4}}(0)}\sk^4dy&\leq& c\left(\int_{B_{\frac{1}{4}}(0)}\sk^{\frac{2N}{N-2}}dy\right)^{\frac{2(N-2)}{N}}\nonumber\\
 				&\leq&c\left(\int_{B_{\frac{1}{4}}(0)}|\nabla\sk|^2+\int_{B_{\frac{1}{4}}(0)}|\sk|^2dy\right)^2\nonumber\\
 						&\leq& c\ \ \mbox{for each $\tau\in [-\frac{1}{32}, \frac{1}{32}]$.}\label{rt55}
 		\end{eqnarray}
 		We finally arrive at
 		\begin{equation}
 		\int_{Q_{\frac{1}{4}}(0)}|n_k|^2\sk^4dyd\tau\leq c.
 		\end{equation}
 	Recall that $q>\frac{N}{2}$.
 Then we have $\frac{2Nq}{(N+2)q-2N}\leq \frac{2N}{N-2}$. Keeping this in mind, we calculate from (\ref{jf4}) that
 	\begin{eqnarray}
 	\|F_k\sk\xi\|_{\frac{2N}{N+2}}^2&\leq& \|F_k\|_{q, B_1(0)}^{2}\|\sk\xi\|_{\frac{2Nq}{(N+2)q-2N}}^2\nonumber\\
 	&\leq& c\|\sk\xi\|_{\frac{2N}{N-2}}^2\nonumber\\
 	&\leq& c\int_{B_1(0)}|\nabla\sk|^2\xi^2dy+c\int_{B_1(0)}|\sk|^2|\nabla\xi|^2dy.\label{rt44}
 	\end{eqnarray}
 	The last term in (\ref{jf1}) can be estimated as follows
 	\begin{eqnarray}
 	\int_{B_1(0)}|F_k||\sk|^3\xi^2 dy &\leq& \|F_k\sk\xi\|_{\frac{2N}{N+2}}\|\sk^2\xi\|_{\frac{2N}{N-2}}\nonumber\\
 	&\leq& c\|F_k\sk\xi\|_{\frac{2N}{N+2}}\|\nabla(\sk^2\xi)\|_{2}\nonumber\\
 	&\leq& \delta\|\nabla(\sk^2\xi)\|_{2}^2+c(\delta)\|F_k\sk\xi\|_{\frac{2N}{N+2}}^2,\ \ \ \delta>0.
 	\end{eqnarray}
 	Substituting this and (\ref{jf2}) into (\ref{jf1}) and choosing $\delta$ suitably small in the resulting inequality yield
 	\begin{eqnarray}
 	\lefteqn{	\int_{B_1(0)}|\sk\nabla\sk|^2\xi^2 dy+\int_{B_1(0)}\nsk^2\sk^2\xi^2 dy}\nonumber\\
 	&\leq&c\int_{B_1(0)}\sk^4|\nabla\xi|^2 dy+c\int_{B_1(0)}|n_k|^2w_k^4|\nabla\xi|^2 dy+c\|F_k\sk\xi\|_{\frac{2N}{N+2}}^2.\label{jf3}
 		\end{eqnarray}
Integrate this inequality over $[-\frac{1}{128}, \frac{1}{128}]$, then choose $\xi$ suitably, i.e., $\xi=1$ on $Q_{\frac{1}{8}}(0)$ and $0$ outside $Q_{\frac{1}{4}}(0)$, and thereby obtain
the claim.
\end{proof}
		Fix $ K>0$. Define
		$$v_k=\left(\wks-K^2\right)^++K^2.$$ 

	\begin{clm}There is a $\beta>0$ such that
		\begin{equation}
		\max_{\tau\in [-\frac{1}{512},\frac{1}{512}]}\int_{B_{\frac{1}{16}}(0)}\int_{0}^{\wks}[(s-K^2)^++K^2]^\beta dsdy\leq c.
			\end{equation}
		\end{clm}
	Obviously, this claim implies (\ref{pb}).	
	\begin{proof} Let $\xi$ be given as in (\ref{rt11})-(\ref{rt22}) and $\beta>0$.
	We may assume that $w_k\in L^\infty(\ot)$ for each $k$. (Otherwise, we use the cut-off function in (\ref{cutoff}).) Then the function $v_k^\beta w_k\xi^2$ is a legitimate test function for (\ref{bl1}).
	Upon using it, we derive
	\begin{eqnarray}
\lefteqn{\frac{1}{2}	\int_{B_1(0)}v_k^\beta\xi^2\partial_\tau\wks dy+D^2\int_{B_1(0)}v_k^\beta\xi^2|\nabla w_k|^2 dy+\frac{D^2\beta}{2}\int_{B_1(0)}v_k^{\beta-1}\xi^2|\nabla v_k|^2 dy}\nonumber\\
&&+D^2\int_{B_1(0)}v_k^\beta \nabla w_k w_k 2\xi\nabla\xi dy+\frac{r_k^2}{\lambda_k}\int_{B_1(0)}|n_k|^{2(\gamma-1)}n_kv_k^\beta w_k\xi^2 dy\nonumber\\
&=&E^2\lambda_k\int_{B_1(0)}\nsk\nabla\sk v_k^\beta w_k\xi^2dy.\label{jf5}
	\end{eqnarray}
	Note that 
	\begin{eqnarray}
	\frac{r_k^2}{\lambda_k}\int_{B_1(0)}|n_k|^{2(\gamma-1)}n_kv_k^\beta w_k\xi^2 dy&=&\frac{r_k^2}{\lambda_k^2}\int_{B_1(0)}|n_k|^{2(\gamma-1)}n_kv_k^\beta (n_k-\mzk)\xi^2 dy\nonumber\\
	&=&\frac{r_k^2}{\lambda_k^2}\int_{B_1(0)}|n_k|^{2\gamma}v_k^\beta \xi^2 dy\nonumber\\
	&&-\frac{r_k^2}{\lambda_k^2}\int_{B_1(0)}|n_k|^{2(\gamma-1)}n_kv_k^\beta \mzk\xi^2 dy\nonumber\\
	&\geq& \frac{1}{2}\frac{r_k^2}{\lambda_k^2}\int_{B_1(0)}|n_k|^{2\gamma}v_k^\beta \xi^2 dy-c\frac{r_k^2}{\lambda_k^2}\int_{B_1(0)}v_k^\beta \xi^2 dy.
	\end{eqnarray}
Now we analyze the last term in (\ref{jf5}) to obtain
\begin{eqnarray}
\lambda_k\int_{B_1(0)}\nsk\nabla\sk v_k^\beta w_k\xi^2dy&=&\int_{B_1(0)}\nsk\nabla\sk v_k^\beta (n_k-\mzk)\xi^2dy\nonumber\\
&=&\int_{B_1(0)}\nsk^2 v_k^\beta \xi^2dy\nonumber\\
&&-\int_{B_1(0)}\nsk\nabla\sk v_k^\beta \mzk\xi^2dy\nonumber\\
&\leq&2\int_{B_1(0)}\nsk^2 v_k^\beta \xi^2dy\nonumber\\
&&+c\int_{B_1(0)}|\nabla\sk|^2 v_k^\beta \xi^2dy.
\end{eqnarray}
Combining the preceding three estimates gives
\begin{eqnarray}
\lefteqn{\frac{1}{2}\frac{d}{d\tau}	\int_{B_1(0)}\int_{0}^{\wks}[(s-K^2)^++K^2]^\beta ds\xi^2 dy+D^2\int_{B_1(0)}v_k^\beta\xi^2|\nabla w_k|^2 dy}\nonumber\\
&&+\frac{D^2\beta}{2}\int_{B_1(0)}v_k^{\beta-1}\xi^2|\nabla v_k|^2 dy+\frac{r_k^2}{\lambda_k^2}\int_{B_1(0)}|n_k|^{2\gamma}v_k^\beta \xi^2 dy\nonumber\\
&\leq&c\int_{B_1(0)}\int_{0}^{\wks}[(s-K^2)^++K^2]^\beta ds\xi\partial_\tau\xi dy+c\int_{B_1(0)}v_k^\beta\wks |\nabla\xi|^2 dy\nonumber\\
&&+c\frac{r_k^2}{\lambda_k^2}\int_{B_1(0)}v_k^\beta \xi^2 dy+2E^2\int_{B_1(0)}\nsk^2 v_k^\beta \xi^2dy+c\int_{B_1(0)}|\nabla\sk|^2 v_k^\beta \xi^2dy.
\label{jf6}
\end{eqnarray}
To estimate the last two terms in the above inequality, we use $\sk v_k^\beta\xi^2$ as a test function in (\ref{bl0}) to obtain
\begin{eqnarray}
\lefteqn{\int_{B_1(0)}|\nabla\sk|^2 v_k^\beta \xi^2dy+\int_{B_1(0)}\nabla\sk\sk \beta v_k^{\beta-1} \nabla v_k\xi^2dy+\int_{B_1(0)}\nabla\sk\sk  v_k^{\beta} 2\xi\nabla\xi dy}\nonumber\\
&&+\int_{B_1(0)}\nsk^2 v_k^\beta \xi^2dy+\int_{B_1(0)}\nsk n_k\sk \beta v_k^{\beta-1} \nabla v_k \xi^2dy\nonumber\\
&&+\int_{B_1(0)}\nsk n_k\sk  v_k^{\beta} 2 \xi \nabla\xi dy\nonumber\\
&=&\int_{B_1(0)}F_k\sk v_k^\beta\xi^2dy.\label{jf7}
\end{eqnarray}
Observe that
\begin{eqnarray}
\left|\int_{B_1(0)}\nabla\sk\sk \beta v_k^{\beta-1} \nabla v_k\xi^2dy\right|&\leq&\frac{D^2}{16E^2}\beta\int_{B_1(0)}v_k^{\beta-1}\xi^2|\nabla v_k|^2 dy\nonumber\\
&&+c\beta\int_{B_1(0)}v_k^{\beta-1}\xi^2|\nabla \sk\sk|^2 dy\nonumber\\
&\leq&\frac{D^2}{16E^2}\beta\int_{B_1(0)}v_k^{\beta-1}\xi^2|\nabla v_k|^2 dy\nonumber\\
&&+c\beta K^{2(\beta-1)}\int_{B_1(0)}\xi^2|\nabla \sk\sk|^2 dy.
\end{eqnarray}
Here we have used the fact that $v_k\geq K^2$ and $\beta<1$. The fifth integral in (\ref{jf7}) can be estimated as follows.
\begin{eqnarray}
\left|\int_{B_1(0)}\nsk n_k\sk \beta v_k^{\beta-1} \nabla v_k \xi^2dy\right|&\leq&\frac{D^2}{16E^2}\beta\int_{B_1(0)}v_k^{\beta-1}\xi^2|\nabla v_k|^2 dy\nonumber\\
&&+c\beta\int_{B_1(0)}v_k^{\beta-1}\xi^2n_k^2\sk^2\nsk^2 dy.
\end{eqnarray}
Remember
\begin{eqnarray}
|n_k|^2\sk^2 &=& |\lambda_k w_k+\mzk|^2\sk^2\nonumber\\
&\leq &2\lambda_k^2\wks\sk^2+c\sk^2\nonumber\\
&\leq& c\wks+c\sk^2
\end{eqnarray}
and $v_k^{\beta-1}\wks\leq v_k^\beta$. Consequently,
\begin{eqnarray}
\left|\int_{B_1(0)}\nsk n_k\sk \beta v_k^{\beta-1} \nabla v_k \xi^2dy\right|&\leq&\frac{D^2}{16E^2}\beta\int_{B_1(0)}v_k^{\beta-1}\xi^2|\nabla v_k|^2 dy\nonumber\\
&&+c\beta\int_{B_1(0)}v_k^{\beta}\xi^2\nsk^2 dy\nonumber\\
&&+c\beta K^{2(\beta-1)}\int_{B_1(0)}\xi^2\sk^2\nsk^2 dy.
\end{eqnarray}
Using the preceding estimates in (\ref{jf7})
\begin{eqnarray}
\lefteqn{\int_{B_1(0)}|\nabla\sk|^2 v_k^\beta \xi^2dy
+(1-c\beta)\int_{B_1(0)}\nsk^2 v_k^\beta \xi^2dy}\nonumber\\
&\leq&\frac{D^2}{8E^2}\beta\int_{B_1(0)}v_k^{\beta-1}\xi^2|\nabla v_k|^2 dy+c\beta K^{2(\beta-1)}\int_{B_1(0)}\xi^2|\nabla \sk\sk|^2 dy\nonumber\\
&&+c\beta K^{2(\beta-1)}\int_{B_1(0)}\xi^2\sk^2\nsk^2 dy+c\int_{B_1(0)}v_k^\beta\sk^2|\nabla\xi|^2dy\nonumber\\
&&+c\int_{B_1(0)}v_k^\beta\wks|\nabla\xi|^2dy+
\int_{B_1(0)}F_k\sk v_k^\beta\xi^2dy.\label{jf8}
\end{eqnarray}
Plugging this into (\ref{jf6}) and choosing $\beta$ suitably small in the resulting inequality, we obtain
\begin{eqnarray}
\lefteqn{\frac{1}{2}\frac{d}{d\tau}	\int_{B_1(0)}\int_{0}^{\wks}[(s-K^2)^++K^2]^\beta ds\xi^2 dy+D^2\int_{B_1(0)}v_k^\beta\xi^2|\nabla w_k|^2 dy}\nonumber\\
&&+\frac{D^2\beta}{2}\int_{B_1(0)}v_k^{\beta-1}\xi^2|\nabla v_k|^2 dy+\frac{r_k^2}{\lambda_k^2}\int_{B_1(0)}|n_k|^{2\gamma}v_k^\beta \xi^2 dy\nonumber\\
&\leq&c\int_{B_1(0)}\int_{0}^{\wks}[(s-K^2)^++K^2]^\beta ds\xi\partial_\tau\xi dy+c\int_{B_1(0)}v_k^\beta\wks |\nabla\xi|^2 dy\nonumber\\
&&+c\frac{r_k^2}{\lambda_k^2}\int_{B_1(0)}v_k^\beta \xi^2 dy+c\beta K^{2(\beta-1)}\int_{B_1(0)}\xi^2|\nabla \sk\sk|^2 dy\nonumber\\
&&+c\beta K^{2(\beta-1)}\int_{B_1(0)}\xi^2\sk^2\nsk^2 dy+c\int_{B_1(0)}v_k^\beta\sk^2|\nabla\xi|^2dy\nonumber\\
&&+
\int_{B_1(0)}F_k\sk v_k^\beta\xi^2dy.
\label{jf9}
\end{eqnarray}
In view of (\ref{rt33}), (\ref{rt55}), and(\ref{rt44}), if $\beta$ is sufficiently small, we have
\begin{eqnarray*}
\int_{Q_{\frac{1}{2}}(0)}\int_{0}^{\wks}[(s-K^2)^++K^2]^\beta dsdyd\tau&\leq&c,\\
\int_{Q_{\frac{1}{2}}(0)}v_k^\beta\wks dyd\tau&\leq&c,\\
\int_{B_{\frac{1}{2}}(0)}v_k^\beta\sk^2dy&\leq&c \ \ \mbox{for $\tau\in [-\frac{1}{8},\frac{1}{8}]$,}\\
\left|\int_{B_{\frac{1}{2}}(0)}F_k\sk v_k^\beta dy\right|&\leq&c \ \ \mbox{for $\tau\in [-\frac{1}{8},\frac{1}{8}]$.}
\end{eqnarray*}
Integrate (\ref{jf9}) with respect to $\tau$, choose $\xi$ suitably, and remember Claim 1 to yield the desired result. 
\end{proof}

This finishes the proof of Proposition 1.2.
\end{proof}

	  \bigskip
	  
	  \noindent{\bf Acknowledgment:} The first author would like to thank Peter Markowich 
	  for suggesting this problem to him and for some helpful discussions. Both authors are grateful to   Peter Markowich and Jan Haskovec for their careful reading of an earlier version of this manuscript and for their useful comments. The research of JL was partially supported by
	  KI-Net NSF RNMS grant No. 1107291 and NSF grant DMS 1514826.

\end{document}